%% file: main.tex
\documentclass{amsart}

\usepackage[a4paper]{geometry}
\newlength{\plarg}
\usepackage{hyperref}
\usepackage{xcolor}
\hypersetup{
    colorlinks=true,
    linkcolor={blue},
    citecolor={blue},
    urlcolor={blue}}

\usepackage{amsmath, amssymb,latexsym,amsthm,comment}
\usepackage[
backend=biber,
style=alphabetic
]{biblatex}

\usepackage{amsmath,amsfonts,amssymb,amsthm}
\usepackage{latexsym,pdfsync,xcolor,graphicx}
\usepackage{mathtools}
\usepackage[T1]{fontenc}

\usepackage{csquotes}
\usepackage[title]{appendix}
\usepackage{tikz}
\usepackage{appendix}
\usepackage{hyperref}
\usepackage{float}
\usepackage{caption}
\usepackage{subcaption}
\usepackage{array}
\usepackage{hhline}
\usepackage{multirow}
\usepackage{indentfirst}

\usepackage{enumitem}
\newlist{steps}{enumerate}{1}
\setlist[steps, 1]{label = Step \arabic*:}

\usetikzlibrary{arrows.meta}

\DeclareMathOperator{\Fp}{\mathbb{F}_{p}}
\usepackage[english]{babel}
\newtheorem{theorem}{Theorem}
\numberwithin{theorem}{section}
\newtheorem{lemma}[theorem]{Lemma}
\newtheorem{corollary}[theorem]{Corollary}

\newtheorem{proposition}[theorem]{Proposition}

\newtheorem*{proposition2.14}{Proposition 2.14}
\newtheorem*{proposition2.15}{Proposition 2.15}

\theoremstyle{definition}
\newtheorem{definition}[theorem]{Definition}
\newtheorem*{notation}{Notation}

\theoremstyle{remark}
\newtheorem{remark}[theorem]{Remark}

\newcommand{\agl}{\text{AGL}(1,\mathbb{F}_{p})}
\newcommand{\gstar}{G^{*}}
\newcommand{\Cen}{\mathrm{Cen}}
\newcommand{\classc}{\mathcal{ST}(p)}
\newcommand{\classcprime}{\mathcal{ST}'(p)}
\newcommand{\classcprimezero}{\mathcal{ST}_{0}'(p)}

\renewcommand{\hat}{\widehat}

\begin{document}
\title{Non-Split Sharply 2-Transitive Groups of odd positive Characteristic}
\author{Marco Amelio, Simon Andr\'e \& Katrin Tent}
\begin{abstract}
It is well-known that every sharply 2-transitive group of characteristic $3$ splits. Here we construct the first examples of non-split sharply 2-transitive groups in odd positive characteristic $p$, for sufficiently large primes~$p$. Furthermore, we show that any group without 2-torsion can be embedded into a non-split sharply 2-transitive group of characteristic~$p$ for all sufficiently large primes $p$, yielding $2^{\aleph_0}$ many pairwise non-isomorphic countable non-split sharply 2-transitive groups in any sufficiently large characteristic.
\end{abstract}
\date{\today}
\maketitle

\input{Introduction.tex}

\input{Preliminaries.tex}

\section{Hyperbolic metric spaces and group actions}
\label{section hyp spaces and grp act}

In this section we recall some concepts from metric geometry. We start by introducing \emph{length metric spaces} and \emph{quasi-geodesics} in Subsection~\ref{length metric spaces} (for a more comprehensive overview, see for example \cite[Chapter 2]{burago_burago_ivanov}). In Subsection~\ref{subsubsection hyperb metric spaces} we recall some of the basic ideas of hyperbolic spaces in the sense of Gromov, and in Subsection~\ref{subsubsection isom of hyp spaces} we study the actions by isometries of groups on these spaces. We follow closely the exposition by Coulon in \cite[Sections 2 and 3]{coulon_1}.

From now on, for a metric space $X$ and two points $x$ and $x'$ of $X$, we will write $d_{X}(x,x')$ (or just $d(x,x')$ if the metric space is clear from the context) for the distance between $x$ and $x'$, unless different notation is explicitly defined.

\subsection{Length metric spaces and quasi-geodesics}
\label{length metric spaces}

For a topological space $X$, a \emph{path} in $X$ is a continuous map $\gamma : I \longrightarrow X$, where $I$ is an interval (possibly consisting of a single point) in $\mathbb{R}$.

\begin{definition}
\label{def length fun by metric}
    Let $(X, d)$ be a metric space and let $\gamma: [a,b] \longrightarrow X$ be a path. A \emph{partition} of $[a,b]$ is a finite collection $Y=\{ y_{0}, \dots y_{N} \}\subset [a,b]$ such that $a=y_{0}\leq y_{1} \leq \dots \leq y_{N}=b$. The \emph{sum of a partition $Y$} is defined as \[ \Sigma(Y)= \sum_{i=1}^{N} d(\gamma(y_{i-1}),\gamma(y_{i})). \]The length of the path $\gamma$ with respect to the metric $d$ is denoted by $L_{d}(\gamma)$ and defined as \[ L_{d}(\gamma)= \sup \{ \Sigma(Y) \, : \, Y \text{is a partition of} \ [a,b] \}. \]
\end{definition}

\begin{definition}
\label{length induces metric}
    Let $(X,d)$ be a metric space connected by rectifiable paths (i.e. between any two points there is a path of finite length). Then the \emph{intrinsic metric} $d_{\ell}$ on $X$ is defined  by putting 
    \[ d_{\ell}(x,y)= \inf \{ L_d(\gamma) \, : \, \gamma:[a,b] \longrightarrow X, \, \gamma(a)=x, \, \gamma(b)=y \}. \]
\label{definition length space}A metric space $(X,d)$ is said to be a \emph{length metric space} (or just a \emph{length space}) if  $d_{\ell}=d$. If moreover $(X,d)$ has the property that there always exists a path $\gamma$ that achieves the infimum in the definition of $d_{\ell}$, then $(X,d)$ is called a \emph{geodesic metric space}, and the path $\gamma$ is called a \emph{geodesic}.
\end{definition}

We now recall the concepts of a \emph{quasi-isometric embedding} and of a \emph{quasi-geodesic}.

\begin{definition}
\label{def quasi-isometry}
    Let $\ell,L \geq 0$ and $k \geq 1$, and let $X_{1}$ and $X_{2}$ be two metric spaces. A map $f: X_{1} \longrightarrow X_{2}$ is a \emph{$(k,\ell)$-quasi-isometric embedding} if for every two points $x,y \in X_{1}$ we have:\[ \frac{1}{k}d_{X_{2}}(f(x),f(y))-\ell \leq d_{X_{1}}(x,y) \leq kd_{X_{2}}(f(x),f(y))+\ell .\] We say that $f$ is an \emph{$L$-local $(k,\ell)$-quasi-isometric embedding} if its restriction to any subset of diameter at most $L$ is a $(k,\ell)$-quasi-isometric embedding. 
\end{definition}

\begin{definition}
\label{def quasi-geodesic}
    Let $X$ be a metric space, $I \subseteq \mathbb{R}$ an interval. A path $\gamma: I \longrightarrow X$ that is a $(k,\ell)$-quasi-isometric embedding is called a \emph{$(k,\ell)$-quasi-geodesic}. If it is an $L$-local $(k,\ell)$-quasi-isometric embedding it is called an \emph{$L$-local $(k,\ell)$-quasi-geodesic}. 
\end{definition}

\begin{remark}Note that for $k=1$ and $\ell=0$, $f$ is a genuine isometry in Definition~\ref{def quasi-isometry}, and $\gamma$ is a genuine geodesic in Definition~\ref{def quasi-geodesic}.
\end{remark}

\begin{remark}\label{existence of qg in length}If $(X,d)$ is a length space, then by definition of the infimum, for every pair of points $x,y \in X$ and every $\ell>0$, there exists a path $\gamma : [a,b]\rightarrow X$ such that $d(x,y)\leq L_d(\gamma)\leq d(x,y)+\ell$. After reparametrizing $\gamma$ (by arc length) if necessary, we can assume that $a=0$ and $b=L_d(\gamma)$ and thus that $L_d(\gamma)=\vert b-a\vert$. Hence, $\gamma$ is a $(1,\ell)$-quasi-geodesic (and thus it is a $(k,\ell)$-quasi-geodesic for every $k\geq 1$).
\end{remark}

\subsection{Hyperbolic metric spaces}
\label{subsubsection hyperb metric spaces}

 For a metric space $X$, a point $x$ in $X$ and a subset $Y$ of $X$, we will write \[ d_{X}(x,Y)= \inf_{y \in Y} \{ d_{X}(x,y) \} \] for the distance between $x$ and $Y$. 
 Also, for a subset $Y$ of $X$ we will write $\text{diam}(Y)$ for the diameter of $Y$, that is, \[ \text{diam}(Y)=\sup_{y,y' \in Y}(d_{X}(y,y')). \] We will put $B_{X}(x,r)$ (or simply $B(x,r)$ if the metric space $X$ is clear by context) for the ball of radius $r$ centered at $x$.

\begin{definition}
\label{gromov prod def}
    Let $x$, $y$ and $z$ be three points of $X$. The \emph{Gromov product of $x$ and $y$ with respect to $z$} is \[\langle x,y \rangle_{z}=\frac{1}{2} \{ d(x,z)+d(y,z)-d(x,y) \}.\]
\label{delta hyp def}A metric space $X$ is said to be \emph{$\delta$-hyperbolic} (in the sense of Gromov) if for every four points $x,y,z,t \in X$ we have \[\langle x,z \rangle_{t} \geq \min \{ \langle x,y \rangle_{t} , \langle y,z \rangle_{t} \} - \delta.\] We will say that $X$ is \emph{hyperbolic} if it is $\delta$-hyperbolic for some $\delta \geq 0$.
\end{definition}

\begin{remark}
  For simplicity of notation, from now on we assume that the hyperbolicity constant $\delta$ is positive. However, notice that this is not a serious restrain: if $X$ is $\delta$-hyperbolic for $\delta \geq 0$, then it is $\delta '$-hyperbolic for every $\delta ' \geq \delta$.
   
    If $X$ is a $\delta$-hyperbolic geodesic metric space, then every geodesic triangle in $X$ is \emph{$2\delta$-thin}, that is, for every geodesic triangle in $X$, every side of the triangle is contained in the closed $2\delta$-neighbourhood of the union of the other two sides \cite[Lemma 11.28]{drutu_kapovich}.
\end{remark}

\subsubsection{Quasi-geodesic polygons} 

We  now collect some useful properties of \emph{quasi-geodesic polygons} in length spaces, similar to the $\delta$-thin condition for triangles in geodesic spaces. We follow closely the exposition by Druţu and Kapovich in \cite[Chapter 11]{drutu_kapovich}.

\begin{definition}
\label{quasi-geodesic triangles}
    Let $X$ be a metric space, $x,y,z \in X$. A \emph{$(k,\ell)$-quasi-geodesic triangle with vertices $x$, $y$ and $z$} is the union of the image of three paths $\gamma_{x,y}$, $\gamma_{y,z}$ and $\gamma_{z,x}$ such that the initial point of $\gamma_{t,t'}$ is $t$, the endpoint is $t'$ and each path is a $(k,\ell)$-quasi-geodesic.
\end{definition}

From now on, we will not distinguish between paths (that are actually maps from an interval of the real line to $X$) from their images in $X$, and we will assume these paths are parametrized by arc length. Notice that if $X$ is a length space, we get immediately from Remark~\ref{existence of qg in length} that for any $x,y,z \in X$ and any $\ell >0$, there exists a $(1,\ell)$-quasi-geodesic triangle with vertices $x$, $y$ and $z$. We will denote one such triangle by $[x,y,z]_{\ell}$ and its sides by $[x,y]_{\ell}$, $[y,z]_{\ell}$ and $[x,z]_{\ell}$ respectively.

In a similar way, we can define a \emph{$(k,\ell)$-quasi-geodesic $n$-gon} for $n \geq 4$. As in geodesic spaces, one such $n$-gon is said to be \emph{$\alpha$-thin} if every side is contained in the closed $\alpha$-neighbourhood of the other $n-1$ sides.

Lemma~\ref{almost uniqueness of quasi-geod} below shows how, even though quasi-geodesics between two points may not be unique in a hyperbolic length space, they cannot go too far away from one another. This result is well-known in the context of geodesic spaces. Since we couldn't find a proof in the literature for length spaces, we include a proof of this result in Appendix~\ref{appendix}.

\begin{lemma}
\label{almost uniqueness of quasi-geod}
    Let $X$ be a $\delta$-hyperbolic metric space, let $x,y \in X$. Let $\gamma$ be a $(1,\ell)$-quasi-geodesic connecting $x$ to $y$ for $\ell \geq 0$. Let $p \in X$ be a point such that
    \begin{equation*}
        d(x,p)+d(p,y) \leq d(x,y) + 2\delta+3\ell.   
    \end{equation*}
    Then, $d(p, \gamma) \leq 4\delta + 8\ell$.
\end{lemma}

\begin{lemma}
\label{thin qg triangles}
    Let $X$ be a $\delta$-hyperbolic metric space and let $[x,y,p]_{\ell}$ be a $(1,\ell)$-quasi-geodesic triangle. Then, this triangle is $(4\delta + 8\ell)$-thin.
\end{lemma}

\begin{proof}
    Let $z \in [x,y]_{\ell}$, we need to prove that it is $(4\delta + 8\ell)$-close to $[y,p]_{\ell} \cup [x,p]_{\ell}$. By $\delta$-hyperbolicity, we have \[ \langle x,y \rangle_{p} \geq \min (\{ \langle x,z \rangle_{p} \, , \, \langle y,z \rangle_{p} \}) - \delta. \] Without loss of generality, we may consider that \[ \langle x,y \rangle_{p} \geq  \langle x,z \rangle_{p} - \delta,\]i.e.\ that  \[ d(y,p)-d(x,y) \geq  d(z,p)-d(x,z) - 2\delta.\]
We will show that $z$ is $(4\delta + 8l)$-close to $[y,p]_{\ell}$. By combining the last inequality with the fact that, since $z \in [x,y]_{\ell}$ we have \[ d(x,y)-d(x,z) \geq d(y,z) - 3\ell, \] we obtain \[ d(y,z)+d(z,p) \leq d(y,p)+3\ell+2\delta. \] By Lemma~\ref{almost uniqueness of quasi-geod}, we obtain that $d(z,[y,p]_{\ell}) \leq 4\delta + 8\ell$.
\end{proof}

Notice that the bound shown may not be the optimal one. We immediately obtain the following corollary.

\begin{corollary}
\label{thin qg polygons}
    Let $X$ be a $\delta$-hyperbolic length space. Then every $(1,\ell)$-quasi-geodesic $n$-gon is $(4(n-2)\delta + 8(n-2)\ell)$-thin.
\end{corollary}

Corollary~\ref{thin qg polygons} and the following lemma are proved in a completely analogous way to the case of geodesic $n$-gons in geodesic spaces (see for example \cite[Section 11.1]{drutu_kapovich}), that is, by triangulating the $n$-gon by $n-3$ $(1,\ell)$-quasi-geodesic diagonals joining a given vertex to the non-adjacent vertices of the $n$-gon.

\begin{lemma}
\label{distances within quadrangles}
    Let $X$ be a $\delta$-hyperbolic length space, $[p,q,r,s]_{\ell}$ a $(1,\ell)$-quasi-geodesic quadrangle with $d(p,q)=d(r,s)$. Then, for any pair of points $x \in [p,q]_{\ell}$ and $y \in [r,s]_{\ell}$ with $d(p,x)=d(s,y)$ we have that \[ d(x,y) \leq 5 \max (\{ d(s,p) \, , \, d(q,r) \}) + 16 \delta + 38\ell. \]
\end{lemma}

\subsubsection{Quasi-convex and strongly quasi-convex subsets} 

We now define the concepts of a \emph{quasi-convex} subset of a metric space and of a \emph{strongly quasi-convex} subset of a hyperbolic length space. For a more comprehensive overview, see \cite[Subsection 2.3]{coulon_1}.

For a subset $Y$ of a metric space $X$, we  write $Y^{\alpha}$ (respectively, $Y^{+\alpha}$) for the open (respectively, closed) $\alpha$-neighbourhood of $Y$.

\begin{definition}
\label{quasi convex def}
    Let $X$ be a metric space, $\alpha \geq 0$. A subset $Y$ of $X$ is \emph{$\alpha$-quasi-convex} if for every pair of points $y,y' \in Y$ and every point $x \in X$ we have that \[ d(x,Y) \leq \langle y,y' \rangle_{x} + \alpha. \]
\end{definition}

\begin{remark}
    If $X$ is a geodesic space, the usual definition of an $\alpha$-quasi-convex subset $Y$ is that every geodesic joining two points of $Y$ is contained in $Y^{+\alpha}$. If $X$ is a $\delta$-hyperbolic geodesic space, a subset is $\alpha$-quasi-convex in the usual sense if and only if it is $(\alpha + 4\delta)$-quasi-convex in the sense of Definition~\ref{quasi convex def}. 
\end{remark}

\begin{definition}
\label{strong quasi convex def}
    Let $X$ be a $\delta$-hyperbolic length space, $\alpha \geq 0$. Let $Y$ be a subset of $X$ connected by rectifiable paths. Denote by $d_{Y}$ the length metric on $Y$ induced by the restriction of the length structure on $X$ to $Y$ (see \cite[Section 2]{burago_burago_ivanov} for the precise definition of a length structure). The subset $Y$ is said to be \emph{strongly quasi-convex} if it is $2\delta$-quasi-convex and for every pair of points $y,y' \in Y$ we have that \[ d_{Y}(y,y') \leq d_{X}(y,y') + 8 \delta . \]
\end{definition}

\begin{definition}
\label{definition hull}
    Let $X$ be a $\delta$-hyperbolic length space, $Y$ a subset of $X$. The \emph{hull of $Y$}, denoted by $\text{hull}(Y)$, is the union of all $(1,\delta)$-quasi-geodesics joining two points of $Y$.
\end{definition}

\begin{lemma}\cite[Lemma 2.15]{coulon_2}
\label{hull is quasi-convex}
    Let $X$ be a $\delta$-hyperbolic length space, $Y$ a subset of $X$. The hull of $Y$ is $6\delta$-quasi-convex.
\end{lemma}

\subsubsection{The boundary at infinity} 

Let $X$ be a $\delta$-hyperbolic metric space, and  $x \in X$. A sequence $(y_{n})_{n \in \mathbb{N}}$ is said to \emph{converge to infinity} if $\langle y_{m},y_{m'} \rangle_{x}$ tends to infinity as $m$ and $m'$ tend to infinity. Note that by hyperbolicity this does not depend on the choice of $x$. The set $\mathcal{S}$ of sequences converging to infinity is endowed with a relation $R \subseteq \mathcal{S}^{2}$ defined as follows: two sequences $(y_{n})_{n \in \mathbb{N}}$ and $(z_{n})_{n \in \mathbb{N}}$ in $\mathcal{S}$ are related if \[ \lim_{n \to \infty} \langle y_{n},z_{n} \rangle_{x}= + \infty. \] Again hyperbolicity gives that this is in fact an equivalence relation.

\begin{definition}
    Let $X$ be a hyperbolic metric space, $\mathcal{S}$ the set of sequences of points of $X$ converging to infinity. The \emph{boundary at infinity of $X$}, denoted as $\partial X$, is the quotient of $\mathcal{S}$ by the equivalence relation $R$.
\end{definition}

This definition does not depend on the choice of the base point $x$ since $X$ is hyperbolic. We will write $[(x_{m})_{m \in \mathbb{N}}]$ for the equivalence class of the sequence $(x_{m})_{m \in \mathbb{N}}$. For a subset $Y$ of $X$, we will write $\partial Y$ for the set of elements of $\partial X$ that are limits of sequences of points of $Y$.

By construction, if a group $G$ acts by isometries on a hyperbolic space $X$ this action extends in a natural way to an action on the boundary $\partial X$: for $ \eta = [(x_{m})_{m \in \mathbb{N}}]\in\partial X$, put $g \cdot \eta = [(g \cdot x_{m})_{m \in \mathbb{N}}] $.

\subsection{Group actions on hyperbolic spaces}
\label{subsubsection isom of hyp spaces}

 For a group $G$ acting by isometries on a hyperbolic metric space $X$, we denote by $\partial G$ the set of accumulation points of $G \cdot x$ in $\partial X$ (note again that this definition does not depend on the choice of $x$). Then either one (and hence every) orbit of $G$ is bounded or $\partial G$ is non-empty (see for example \cite[Proposition 3.4]{coulon_1}).
 
Recall that if $g$ is an isometry of $X$, then $g$ is of one of the following types: 
    \begin{itemize}
        \item elliptic, i.e. $\partial \langle g \rangle$ is empty,
        \item parabolic, i.e.  $\partial \langle g \rangle$ has exactly one element, and
        \item loxodromic, i.e.  $\partial \langle g \rangle$ has exactly two elements.
    \end{itemize}

For a loxodromic isometry $g$, the two elements of $\partial \langle g \rangle$ are \[ g^{- \infty} = [(g^{-m} \cdot x)_{m \in \mathbb{N}}] \, \, \text{and} \, \, g^{+ \infty} = [(g^{m} \cdot x)_{m \in \mathbb{N}}]. \] 
 They are the only points of $\partial X$ fixed by $g$ (see \cite[Chapitre 10, Proposition 6.6]{coornaert_delzant_papadopoulos}). Conversely we have the following well-known lemma.
 
\begin{lemma} {\cite[Proposition 3.6]{coulon_1}}
\label{two points in boundary implies loxodromic}
    Let $G$ be a group acting by isometries on a hyperbolic metric space $X$. If $\partial G$ has at least two points, then $G$ contains a loxodromic isometry.
\end{lemma}

Next we introduce two concepts of translation lengths that can be used, among other things, to give a characterization of loxodromic isometries.

\begin{definition}
\label{translation length definition}
    Let $X$ be a hyperbolic metric space, $g$ an isometry of $X$. The \emph{translation length} of $g$, denoted by $[g]_{X}$ (or simply $[g]$ if the metric space $X$ is clear from the context) is \[ [g]_{X}= \inf \{ d(x, g \cdot x) \, : \, x \in X \}. \] The \emph{asymptotic translation length} of $g$, denoted by $[g]^{\infty}_{X}$ (or simply $[g]^{\infty}$) is \[ [g]^{\infty}_{X}= \lim_{n \to + \infty} \frac{1}{n}d(x, g^{n} \cdot x).\]
\end{definition}

Once again, notice that the definition of the asymptotic translation length does not depend on the choice of $x$. These two concepts are related as follows:

\begin{lemma} {\cite[Chapitre 10, Propositions 6.3 and 6.4]{coornaert_delzant_papadopoulos}}
\label{asympt and trans length relation}
    Let $X$ be a $\delta$-hyperbolic metric space, $g$ an isometry of $X$. Then, the quantities $[g]$ and $[g]^{\infty}$ satisfy \[ [g]^{\infty} \leq [g] \leq [g]^{\infty} + 32 \delta, \]
 and $g$ is loxodromic if and only if $[g]^{\infty} > 0$.
\end{lemma}

\subsubsection{The axis of an isometry}
\label{subsubsection axis of isom}

We  now introduce the concepts of the \emph{axis} of an isometry and the \emph{cylinder} of a loxodromic isometry, which will play an important role in the small cancellation results introduced in Section~\ref{subsection geometric small cancellation}.
For a hyperbolic length space $X$ and two distinct points $\zeta$ and $\eta$ of $\partial X$, we say that a path $\gamma: \mathbb{R} \longrightarrow X$ joins $\zeta$ and $\eta$ if \[ \{ [(\gamma(-m))_{m \in \mathbb{N}}], [(\gamma(m))_{m \in \mathbb{N}}] \} = \{ \zeta, \eta \}. \] 


\begin{definition}
\label{def axis of isometry}
    Let $X$ be a hyperbolic metric space, $g$ be an isometry of $X$. The \emph{axis} of $g$, denoted as $A_{g}$ is the set \[ \{ x \in X \, : \, d(x, g \cdot x) < [g] + 8 \delta \}. \]
\end{definition}

Note that the axis is defined for any isometry of $G$, not necessarily a loxodromic one.

In the following, let $L_{S}\in\mathbb{N}$ be such that  every $L_{S}\delta$-local $(1,10^{5}\delta)$-quasi-geodesic is a (global) $(2,10^{5}\delta)$-quasi-geodesic, and there is a bound for the Hausdorff distance between any two such quasi-geodesics. 

Such a number $L_S$ exists and does not depend on $\delta$ by \cite[Proposition 2.6 and Corollary 2.7]{coulon_1}.

\begin{definition}
\label{def cylinder of isometry}
    Let $X$ be a $\delta$-hyperbolic length space, $g$ a loxodromic isometry of $X$. We denote by $\Gamma_{g}$ the union of all $L_{S} \delta$-local $(1, \delta)$-quasi-geodesics joining $g^{- \infty}$ and $g^{+ \infty}$. The \emph{cylinder} of $g$, denoted as $Y_{g}$ is the open $20\delta$-neighbourhood of $\Gamma_{g}$.
\end{definition}

The next result relates the axis and the cylinder of a loxodromic isometry (see \cite[Lemmas 2.32 and 2.33]{coulon_2} and \cite[Lemma 3.13]{coulon_1}).

\begin{lemma}
\label{axis vs cylinder}
    Let $X$ be a $\delta$-hyperbolic length space, $g$ a loxodromic isometry of $X$, $A_{g}$ the axis of $g$ and $Y_{g}$ the cylinder of $g$. Then $A_{g} \subseteq Y_{g}$, $Y_{g} \subseteq A_{g}^{+52 \delta}$ and $Y_{g}$ is a strongly quasi-convex subset of $X$.
\end{lemma}

\subsubsection{Elementary subgroups}\label{def element subgroup} 

Let $X$ be a hyperbolic metric space, and let $G$ be a group acting on $X$ by isometries. Let $H$ be a subgroup of $G$. We say that $H$ is \emph{elementary} if $\partial H$ has at most two points. Otherwise, we say it is \emph{non-elementary}. We say that an elementary subgroup $H$ is
    \begin{itemize}
        \item \emph{elliptic} if its orbits are bounded (equivalently, if $\partial H$ is empty),
        \item \emph{parabolic} if $\partial H$ has exactly one point, or
        \item \emph{loxodromic} if $\partial H$ has exactly two points.
    \end{itemize}

Notice that any finite subgroup of $G$ is elliptic. We can associate to an elliptic subgroup a set of `almost fixed points' in the sense of the following definition.

\begin{definition}
\label{def charact subset of elliptic}
    Let $F$ be an elliptic subgroup of $G$. The \emph{characteristic set} of $F$ is \[ C_F= \{ x \in X \, : \, \forall g \in F, \, d(g \cdot x,x) \leq 11 \delta \}. \]
\end{definition}

\begin{lemma}\cite[Corollary 3.27]{coulon_1}
\label{charac subset of elliptic is quasi convex}
    Let $F$ be an elliptic subgroup of $G$. The characteristic set $C_F$ is $9 \delta$-quasi-convex.
\end{lemma}

\subsubsection{Acylindrical group actions}

Recall that a group is hyperbolic if and only if it acts properly and cocompactly by isometries on a hyperbolic metric space. In recent years, many interesting results were obtained for groups that admit a (non-elementary) action on a hyperbolic metric space by isometries satisfying a weaker condition. One of the two weaker conditions that will be considered in this paper is \emph{acylindricity}. This notion goes back to Sela’s paper \cite{Sela}, where it was considered for groups acting on trees. In the context of general metric spaces, the following definition was introduced by Bowditch in \cite{bowditch}.


\begin{definition}
\label{acylindr definition}
Let $G$ be a group acting by isometries on a $\delta$-hyperbolic metric space $X$. The action is said to be \emph{acylindrical} if for every $ \varepsilon \geq 0$ there exist $M, L > 0$ such that for every $x,y \in X$ with $d(x,y) \geq L$ we have: \[ \lvert \{ g \in G : d(x,g \cdot x) \leq \varepsilon, \, d(y,g \cdot y) \leq \varepsilon \} \rvert \leq M .\]    
\end{definition}

\begin{remark}
\label{acylind characterization}
   By \cite[Proposition 5.31]{dahmani_guirardel_osin} it suffices to check this condition for $\varepsilon=100\delta$. Even though this result is stated for geodesic spaces, this also holds for length spaces (see \cite[Proposition 5.6]{coulon_4}). 
\end{remark}

The following lemma, which is an immediate consequence of \cite[Theorem 1.1]{osin}, will play an important role in the proof of our main result.

\begin{lemma}\label{no_parabolic_osin}Le $G$ be a group acting acylindrically by isometries on a hyperbolic metric space. Then $G$ has no parabolic subgroup.
\end{lemma}

\subsubsection{Weakly properly discontinuous group actions}

We now define a weakening of acylindricity, namely \emph{weak proper discontinuity}. This notion was introduced by Bestvina and Fujiwara in \cite{bestvina_fujiwara}.

\begin{definition}
\label{definition wpd}
Let $G$ be a group acting by isometries on a metric space $X$, and let $h \in G$ be a loxodromic element. The isometry $h$ is said to satisfy the \emph{weak proper discontinuity property} (WPD property) if for every $x\in X$ and every $ \varepsilon \geq 0$ there exist $n \in \mathbb{N}$, $M>0$ such that \[\lvert \{ g \in G : d(x,g \cdot x) \leq \varepsilon, \, d(h^{n} \cdot x,gh^{n} \cdot x) \leq \varepsilon \} \rvert \leq M . \] The action is said to be \emph{weakly properly discontinuous} (WPD) if every loxodromic element satisfies the WPD property.
    
\end{definition}

\begin{remark}
Clearly, an acylindrical action is also WPD.
\end{remark}


The following lemma will be crucial (see \cite[Section 3]{coulon_1}).

\begin{lemma}\label{lox subg in wpd is virt cyc}Let $G$ be a group with a WPD action by isometries on a hyperbolic length space $X$. Then every loxodromic subgroup $H$ of $G$ is contained in a unique maximal loxodromic subgroup of $G$, namely the setwise stabilizer of the pair $\partial H\subset \partial G$, denoted by $M_G(H)$. Moreover, $M_G(H)$ (and thus $H$) are virtually cyclic.
\end{lemma}

Recall that an infinite virtually cyclic group $H$ maps either onto $\mathbb{Z}$ or onto $D_{\infty}$, with finite kernel (which is the unique maximal normal finite subgroup of $H$). In the first case, we say that $H$ is of \emph{cyclic type}, and in the second case we say that $H$ is of \emph{dihedral type}. An element of $H$ is called \emph{primitive} if it maps to an element of $\mathbb{Z}$ (in the first case) or of $D_{\infty}$ (in the second case) that has infinite order and does not admit a proper root. This terminology can be extended to a loxodromic element $h$ of $G$: the element $h$ is called \emph{primitive} if it is primitive as an element of the virtually cyclic subgroup $M_G(\langle h\rangle)$ (equivalently, $h$ has minimal asymptotic translation length among the loxodromic elements of $M_G(\langle h\rangle)$).

The next lemma relates the cylinder of a loxodromic isometry with the characteristic subset of finite subgroups normalized by this element.

\begin{lemma}\cite[Lemma 3.33]{coulon_1}
\label{charac subs of max finite sgrp}
    Let $G$ be a group with a WPD action by isometries on a hyperbolic length space $X$. Let $g$ be a loxodromic element of $G$ and $H$ a subgroup fixing the set $\{g^{\pm \infty}\}$ pointwise. Let $F$ be the maximal normal finite subgroup of $H$. Then, the cylinder $Y_g$ is contained in the $51 \delta$-neighbourhood of the characteristic subset $C_F$.
\end{lemma}

\section{Geometric small cancellation}
\label{subsection geometric small cancellation}

In the 1910s, Dehn proved that for the fundamental group of a closed orientable surface of genus at least two the word problem is solvable. His work involved negative curvature, and was a precursor for small cancellation theory. Small cancellation conditions were formulated explicitely for the first time by Tartakovskii in 1949. Then, small cancellation theory was developed notably by Greendlinger in the early 1960s and by Lyndon and Schupp around the same time to study groups given by group presentations where defining relations have small overlaps with each other. However, the geometric origins of small cancellation theory have gradually been forgotten in favour of combinatorial and topological methods. According to Gromov, ``the role of curvature was reduced to a metaphor (algebraists do not trust geometry)'', and he proposed to return to the geometric sources of small cancellation theory. This point of view appears in \cite{gromov_mesoscopic} and \cite{delzant_gromov}, and was further developed extensively in the work of Coulon.
 
In this section we introduce and adapt the small cancellation results proved in \cite{coulon_1}. We start by defining in Subsection~\ref{subsubsection invariants} some \emph{invariants} of a WPD group action on a hyperbolic metric space (for a more comprehensive overview, see \cite[Section 3.5]{coulon_1}). Then, in Subsection~\ref{subsubsection cone off} we introduce the \emph{cone-off} construction (see also \cite[Section 4]{coulon_1}), that will play a central role when we state Theorem~\ref{small cancellation theorem} in Subsection~\ref{subsubsections small cancellation theorems} and Proposition~\ref{res: SC - induction lemma} and Theorem~\ref{res : SC - partial periodic quotient} in Subsection~\ref{subsubsection partial periodic and small cancellation}, which are the small cancellation results used to prove Proposition~\ref{prop: classes stable under pp quotients}.

 \subsection{Invariants of a Group Action}
 \label{subsubsection invariants}

 Throughout this subsection we fix a group $G$ with a WPD action by isometries on a hyperbolic space $X$. We introduce the invariants $r_{\text{inj}}(Q,X)$ (for a subset $Q$ of $G$), $e(G,X)$, $\nu(G,X)$ and $A(G,X)$. They will play the role of the small cancellation parameters in Subsection~\ref{subsubsections small cancellation theorems}.

 \begin{definition}
 \label{def rinj}
     Let $Q$ be a subset of $G$. The \emph{injectivity radius} of $Q$ is \[ r_{\text{inj}}(Q,X)= \inf \{ \left[ g\right]^{\infty} : g \in Q, \, \, g \, \, \text{loxodromic} \}. \]
 \end{definition}



\begin{definition}
\label{def parameter e}
   The invariant $e(G,X)$ is the least common multiple of $\lvert F \rtimes \text{Aut}(F) \rvert$ where $F$ runs over the maximal normal finite subgroup of all maximal loxodromic subgroups of $G$.
\end{definition}

\begin{remark}
\label{e is 1 if lox subg are cyc}
    Notice that if all loxodromic subgroups of $G$ are cyclic, then we will get $e(G,X)=1$.
\end{remark}

\begin{definition}
\label{def param nu}
    The invariant $\nu (G,X)$ (or simply $\nu$) is the smallest positive integer $m$ satisfying the following property: let $g$ and $h$ be two isometries of $G$ with $h$ loxodromic. If $g$, $h^{-1}gh$,..., $h^{-m}gh^m$ generate an elementary subgroup which is not loxodromic, then $g$ and $h$ generate an elementary subgroup of $G$.
\end{definition}

The proof of Lemma 6.12 in \cite{coulon_1} yields the following bound for $\nu(G,X)$ for acylindrical actions with positive injectivity radius.

\begin{lemma}
\label{nu finite acyl}
    Assume the action of $G$ on $X$ is acylindrical and with positive injectivity radius. Call $L$ and $M$ the parameters in the definition of an acylindrical action (Definition~\ref{acylindr definition}) corresponding to $\varepsilon = 97\delta$, and put $M'$ as the smallest positive integer such that $M' r_{\text{inj}}(G,X) \geq L$. Then, $\nu(G,X) \leq M'+M$.
\end{lemma}

\begin{definition}
\label{def param A}
    Assume the action of $G$ on $X$ has finite parameter $\nu = \nu(G,X)$. 
    
For  $g_{1}, \dots , g_{m}\in G$ we put \[A(g_{1},\dots, g_{m}) =  \text{diam}\left(A_{g_{1}}^{+13\delta} \cap \dotsc \cap A_{g_{m}}^{+13\delta}\right).\]    
    We denote by $\mathcal{A}$ the set of $(\nu+1)$-tuples $(g_{0},\dots,g_{\nu})$ such that $g_{0},\dots,g_{\nu}$ generate a non-elementary subgroup of $G$ and for all $j \in \{0, \dots, \nu \}$ we have $[g_{j}] \leq L_{S}\delta$. We define \[ A(G,X) = \sup_{(g_{0},\dots,g_{\nu}) \in \mathcal A} (\{A\left(g_{0}, \dots,g_{\nu}\right)\}). \]
		
\end{definition}

\begin{remark}
\label{invariants in rescaled space}
    Notice that if $G$ acts on a $\delta$-hyperbolic metric space $X$ and $\lambda X$ is a rescaling of $X$ (that is, a metric space with the same underlying set and distances multiplied by $\lambda$), then $\lambda X$ is a $\lambda \delta$-hyperbolic metric space endowed with an action of $G$ and invariants $r_{\text{inj}}(Q, \lambda X)= \lambda r_{\text{inj}}(Q, X)$ for any subset $Q$ of $G$, $e(G, \lambda X)=e(G,X)$, $\nu(G, \lambda X)= \nu(G,X)$ and $A(G, \lambda X)= \lambda A(G,X)$.
\end{remark}

\subsection{The Cone-Off Construction}
\label{subsubsection cone off}

In this subsection we introduce the \emph{cone-off} construction over certain families of subspaces of a metric space, and explain how to extend the action of a group on the metric space to an action on the cone-off. This construction will allow us to iteratively apply the Small Cancellation Theorem introduced in Subsection~\ref{subsubsections small cancellation theorems}. For the rest of this section, we fix a number $\rho >0$.

\begin{definition}
\label{definition cone over X}
    Let $X$ be a metric space. The \emph{cone over $X$ of radius $\rho$}, denoted by $Z_{\rho}(X)$ (or, if the value of $\rho$ is clear by context, simply $Z(X)$) is the topological quotient of $X \times [0,\rho]$ by the equivalence relation identifying all points of the form $(x,0)$, $x\in X$.
\end{definition}

The equivalence class of $(x,0)$ is called the \emph{apex} of the cone. The cone over $X$ is endowed with a metric characterized as follows (see \cite[Chapter I.5, Proposition 5.9]{bridson_haefliger}). Let $x=(y,r)$ and $x'=(y',r')$ be two points of $Z(X)$, then \[ \cosh(d_{Z(X)}(x,x'))=\cosh(r)\cosh(r')-\sinh(r)\sinh(r')\text{cos}\left(\min \left( \pi , \dfrac{d_{X}(y,y')}{\sinh(\rho)} \right)\right).\] 

In addition, if $X$ is a length space, then so is $Z(X)$. An action by isometries of a group $G$ over $X$ naturally extends to an action by isometries on $Z(X)$ as follows: for $x=(y,r)$ in $Z(X)$ and $g \in G$, put $g \cdot x=(g \cdot y,r)$. Note that in this case the apex of $Z(X)$ is  a global fixed point.

We can compare the original metric space $X$ with its cone $Z(X)$ by defining a \emph{comparison map} $\psi: X \longrightarrow Z(X)$ such that $x \longmapsto (x, \rho)$.

Now we are ready to introduce the cone-off construction.

\begin{definition}
\label{definition cone off}
    Let $X$ be a hyperbolic length space, $\mathcal{Y}$ a collection of strongly quasi-convex subsets of $X$. For $Y \in \mathcal{Y}$, denote by $d_{Y}$ the metric on $Y$ given by the length structure on $Y$ induced by the length structure of $X$, by $Z(Y)$ the cone over $Y$ (endowed with the distance $d_{Y}$) of radius $\rho$ and by $\psi_{Y}$ the corresponding comparison map.

    The \emph{cone-off of radius $\rho$ over $X$ relative to $\mathcal{Y}$}, denoted by $\dot{X}_{\rho}(\mathcal{Y})$ (or simply $\dot{X}$ if $\rho$ and $\mathcal{Y}$ are clear by context) is the quotient of the disjoint union of $X$ and the $Z(Y)$ for all $Y \in \mathcal{Y}$ by the equivalence relation that, for all $Y \in \mathcal{Y}$ and $y \in Y$, identifies $y$ with $\psi_Y(y) \in Z(Y)$.
\end{definition}

Since the hyperbolic length space $X$ embeds into the cone-off $\dot X$, we will identify it with its image under this embedding. The cone-off is naturally endowed with a metric given by the length structure on $\dot{X}$ induced by the length structures on $X$ and $Z(Y)$ for all $Y \in \mathcal{Y}$ \cite[Section 4.2]{coulon_1}.

 The following lemma gives conditions under which the cone-off is hyperbolic, with certain control over the hyperbolicity constant. For this purpose, we introduce a parameter that controls the overlap between the elements of $\mathcal{Y}$. We write \[ \Delta(\mathcal{Y})= \sup_{Y_{1} \neq Y_{2} \in \mathcal{Y}}\left(\text{diam}\left(Y_{1}^{+5\delta} \cap Y_{2}^{+5\delta}\right)\right). \]
 
 Following Coulon, we denote by $\boldsymbol{\delta}$ the hyperbolicity constant of the hyperbolic plane.

\begin{lemma}\cite[Proposition 6.4]{coulon_2}
\label{cone off hyperb}
    There exist positive numbers $\delta_{0}$, $\Delta_{0}$ and $\rho_{0}$ that satisfy the following property. Let $X$ be a $\delta$-hyperbolic space with $\delta \leq \delta_{0}$. Let $\mathcal{Y}$ be a family of strongly quasi-convex subsets of $X$ with $\Delta(\mathcal{Y}) \leq \Delta_{0}$. Let $\rho \geq \rho_{0}$. Then the cone-off $\dot{X}_{\rho}(\mathcal{Y})$ is $\dot{\delta}$-hyperbolic, with $\dot{\delta}=900 \boldsymbol{\delta}$.
\end{lemma}

For the remainder of this subsection, we fix a length space $X$ and a family $\mathcal{Y}$ as in Definition~\ref{definition cone off}. 

Lemmas~\ref{dist on x compared cone off} and~\ref{dist on Y compared cone off} provide some insight on how the metric on $\dot{X}$ relates to the metric on $X$ and to the cones $Z(Y)$. In order to state the first of these results, we introduce the map $\mu: \mathbb{R}_{\geq 0} \longrightarrow \mathbb{R}_{\geq 0}$ characterized by \[ \cosh(\mu(t))=\cosh^{2}(\rho)-\sinh^{2}(\rho)\text{cos}\left(\min\left( \pi, \dfrac{t}{\sinh(\rho)}\right)\right) \] for all $t \geq 0$. The map $\mu$ has the following properties that will be used further down.

\begin{lemma}\cite[Proposition 4.2]{coulon_1}
\label{bound on mu for small t}
    The map $\mu$ is continuous, concave (down) and non-decreasing. Furthermore, the following hold.
    \begin{itemize}
        \item For all $t \geq 0$, we have: $t-\dfrac{1}{24}\left(1+\dfrac{1}{\sinh^{2}(\rho)}\right)t^{3} \leq \mu(t) \leq t$, and
        \item for all $t \in [0, \pi \sinh(\rho)]$, we have: $t \leq \pi \sinh\left(\dfrac{\mu(t)}{2}\right)$.
    \end{itemize}
\end{lemma}

\begin{lemma}\cite[Lemma 5.8]{coulon_2}
\label{dist on x compared cone off}
    For every $x,x' \in X$ we have: \[ \mu(d_{X}(x,x')) \leq d_{\dot{X}}(x,x') \leq d_{X}(x,x'). \]
\end{lemma}

\begin{lemma}\cite[Lemma 5.7]{coulon_2}
\label{dist on Y compared cone off}
    Let $v$ be the apex of a cone $Z(Y)$ for some $Y \in \mathcal{Y}$. Then, $B_{\dot{X}}(v,\rho)=Z(Y) \backslash Y$.
\end{lemma}

\subsubsection{Group action on the cone-off} For the remainder of Section~\ref{subsection geometric small cancellation}, we fix real numbers $\rho \geq \max  (\{ \rho_{0}, 10^{10}\boldsymbol{\delta},10^{20}L_{S}\boldsymbol{\delta}\})$, $\delta \leq \delta_{0}$, a $\delta$-hyperbolic length space $X$, a family $\mathcal{Y}$ of strongly quasi-convex subsets of $X$ with $\Delta(\mathcal{Y}) \leq \Delta_{0}$, where $\rho_{0}$, $\delta_{0}$ and $\Delta_{0}$ are the parameters provided by Lemma~\ref{cone off hyperb}.


Consider a group $G$ acting by isometries on $X$, and acting on the family $\mathcal{Y}$ by left translation, that is, such that $g \cdot Y \in \mathcal{Y}$ for all $Y \in \mathcal{Y}$. We can extend this action by homogeneity to an action of $G$ on the cone-off, as follows. Let $Y \in \mathcal Y$ and $x = (y,r)$ be a point of the cone $Z(Y)$. For $g\in G$ we define $g\cdot x = g\cdot (y,r) = (g \cdot y,r)\in Z(g \cdot Y)$. It follows from the definition of the metric of $\dot X$ that this action is an isometry on $\dot X$.

The next result uses techniques from \cite[Proposition 4.10]{coulon_1} for WPD actions on hyperbolic length spaces, and from \cite[Proposition 5.40]{dahmani_guirardel_osin} for acylindrical actions on hyperbolic geodesic spaces.

\begin{lemma}
\label{acyl passes to cone-off}
    If the action of $G$ on $X$ is acylindrical, then so is the induced action on $\dot{X}$.
\end{lemma}

\begin{proof}
    We will apply Remark~\ref{acylind characterization}. The action on $X$ is acylindrical, therefore, there are positive numbers $L'$ and $M'$ such that, for all $x,x' \in X$, if $d_{X}(x,x')$ is at least $L'$, then there are at most $M'$ elements moving $x$ and $x'$ less than $\pi \sinh(300\dot{\delta})$. We will show that we can take $M'$ and $L'+4 \rho$ as the parameters $M$ and $L$ of Remark~\ref{acylind characterization}.

    Now, let $a,b \in \dot X$ be such that $d(a,b) \geq L' + 4 \rho$, and consider a $(1, \dot \delta)$-quasi-geodesic segment $[a,b]_{\dot \delta}$. If an element $g \in G$ moves both $a$ and $b$ by less than $100 \dot \delta$, we can apply Lemma~\ref{distances within quadrangles} to conclude that any point in this quasi-geodesic segment is moved by less than $600 \dot \delta$ by $g$. Furthermore, since the diameter of the ball around an apex of the cone is $2\rho$, this quasi-geodesic must contain points of $X$ at distance at least $L'$. Therefore, we may assume that $a,b \in X$ and that we need to bound the number of elements of $G$ moving $a$ and $b$ by at most $600 \dot \delta$. By Lemma~\ref{dist on x compared cone off}, we have that \[\mu(d_{X}(a,g \cdot a)) \leq d_{\dot X}(a,g \cdot a) \leq 600 \dot \delta.\] By the choice of $\rho$, we have that $\mu(d_{X}(a,g \cdot a)) < \pi \sinh(\rho)$, and therefore Lemma~\ref{bound on mu for small t} gives \[ d_{X}(a, g \cdot a) \leq \pi \sinh (300 \dot \delta). \] Similarly, we obtain that \[ d_{X}(b, g \cdot b) \leq \pi \sinh (300 \dot \delta), \] so the number of elements $g$ satisfying this property is indeed at most $M'$.
\end{proof}

\subsection{The Small Cancellation Theorem}
\label{subsubsections small cancellation theorems}

In this subsection we will use the cone-off construction introduced in Subsection~\ref{subsubsection cone off} to state a small cancellation theorem. We follow  the exposition by Coulon in \cite[Section 5]{coulon_1} with the difference that we consider the more restricted case of acylindrical actions (instead of WPD ones). To simplify the statements of some results, throughout this subsection we will fix a group $G$ acting on a $\delta$-hyperbolic length space $X$. We assume that this action is acylindrical and non-elementary. We consider a family $\mathcal{Q}$ of pairs $(H,Y)$ such that
\begin{itemize}
    \item there exists an odd integer $n \geq 100$ such that for all $H$ we have a loxodromic element $h'\in G$ such that $H=\langle h'^{n} \rangle$,
    \item $Y$ is the cylinder $Y_{h}$ for $h=h'^{n}$,
    \item there is a normal subgroup $N \trianglelefteq G$ without involutions and containing the normal subgroup $K$ of $G$ generated by the collection $\{H \, : \, (H,Y) \in \mathcal{Q}\}$ such that for every pair $(H,Y) \in \mathcal{Q}$ the element $h'$ is a primitive loxodromic element of $N$, and
    \item there is an action of $G$ on $\mathcal{Q}$ via  $g \cdot (H,Y)=(g^{-1}Hg,g^{-1}Y)$ for  $g \in G$ and $(H,Y) \in \mathcal{Q}$.
\end{itemize}

We want to study the quotient $\hat{G}=G / K$, and to that purpose, we will define a metric space $\hat{X}$ on which $\hat{G}$ acts. Notice first that, since $Y$ is a strongly quasi-convex subset of $X$ for all $(H,Y) \in \mathcal{Q}$ (see Lemma~\ref{axis vs cylinder}), we can construct the cone-off $\dot X$ of radius $\rho$ relative to the family $\{Y \, : \, (H,Y) \in \mathcal{Q}\}$. The group $G$ has a natural action on this space induced by the action of $G$ on $X$. We define the space $\hat{X}$ as the quotient of $\dot X$ by $K$. This space is endowed with an action of $\hat{G}$ (see \cite[Section 5.1]{coulon_1}). We write $\zeta: \dot X \longrightarrow \hat{X}$ for the projection map and $v(\mathcal{Q})$ for the subset of $\dot X$ consisting of the apices of the cones $Z(Y)$ for $(H,Y) \in \mathcal{Q}$. Then $\hat{v}(\mathcal{Q})$ denotes its image in $\hat{X}$. We will also call the elements of $\hat{v}(\mathcal{Q})$ apices. For an element $g \in G$ (respectively, $x \in \dot X$), we will write $\hat{g}$ (respectively, $\hat{x}$) for its image in $\hat{G}$ (respectively, $\hat{X}$).

In order to get some desired properties of the action of $\hat{G}$ on $\hat{X}$, we need the action of $G$ on $X$ to satisfy some conditions related to the quantities $\Delta(\mathcal{Q})$ and $T(\mathcal{Q})$, that play the role of the length of the largest piece and the length of the shortest relator in the usual small cancellation theory, and that are defined as \[ \Delta(\mathcal{Q})= \sup(\{ \text{diam}(Y_{1}^{+5\delta} \cap Y_{2}^{+5\delta}) \, : \, (H_{1},Y_{1}) \neq (H_{2},Y_{2}) \in \mathcal{Q}\}) \] and \[ T(\mathcal{Q})= \inf (\{ [h] \, : \, h \in H \, , \, (H,Y) \in \mathcal{Q} \}). \]

Now we can introduce the aforementioned small cancellation result. Recall that $\boldsymbol{\delta}$ denotes the hyperbolicity constant of the hyperbolic plane.

\begin{lemma}
\label{small cancellation theorem}
There exist positive constants $\rho_{0}$, $\delta_{0}$ and $\Delta_{0}$ that are independent of $X$, $G$ and $\mathcal{Q}$ such that for $\delta \leq \delta_{0}$, $\rho \geq \rho_{0}$, $\Delta(\mathcal{Q}) \leq \Delta_{0}$ and $T(\mathcal{Q}) \geq 8\pi \sinh(\rho)$ the following hold.
    \begin{enumerate}
        \item(see \cite[Proposition 6.4]{coulon_2}) The cone-off $\dot X$ is a $\dot \delta$-hyperbolic length space with $\dot \delta = 900 \boldsymbol{\delta}$.
        \item(see \cite[Proposition 6.7]{coulon_2}) The space $\hat{X}$ is a $\hat{\delta}$-hyperbolic length space with $\hat{\delta}= 64 \cdot 10^{4} \boldsymbol{\delta}$.
        \item(see \cite[Propositions 5.14 and 5.15]{coulon_2}) The group $\hat{G}=G/K$ acts by isometries on $\hat{X}$, and this action is WPD and non-elementary.
        \item For every $(H,Y) \in \mathcal{Q}$, the projection $G \twoheadrightarrow \hat{G}$ induces an isomorphism from $\text{\text{Stab}}(Y) / H$ onto the image of $\text{Stab}(Y)$.
        \label{small cancellation theorem isom of staby}

\item(see \cite[Proposition 5.16]{coulon_1}) \label{g hat preserves elem}The image in $\hat{G}$ of an elliptic (respectively parabolic, loxodromic) subgroup of $G$ is elliptic (respectively parabolic or elliptic, elementary).
 \item (see \cite[Proposition 5.17]{coulon_1}) \label{g hat isomorph of elliptic}Let $E$ be an elliptic subgroup of $G$. The projection $G \twoheadrightarrow \hat{G}$ induces an isomorphism from $E$ onto its image.
\item(see \cite[Proposition 5.18]{coulon_1}) \label{elliptic subgroups of hat g}Let $\hat{E}$ be an elliptic subgroup of $\hat{G}$. Then one of the following holds.
    \begin{enumerate}
        \item There is an elliptic subgroup $E$ of $G$ (for its action on $X$) such that the projection $G \twoheadrightarrow \hat{G}$ induces an isomorphism from $E$ onto $\hat{E}$.
        \item There exists $\hat{v} \in \hat{v}(\mathcal{Q})$ such that $\hat{E}$ is contained in $\text{\text{Stab}}(\hat{v})$.
    \end{enumerate}
    \end{enumerate}
\end{lemma}

\begin{remark}
\label{stab of v in the small cancellation theorem}
    Notice that $(H_{1},Y_{1}) \neq (H_{2},Y_{2}) \in \mathcal{Q}$ implies $Y_{1}\neq Y_{2}$ since $\Delta(\mathcal{Q})$ is finite and the cylinder of a loxodromic element is unbounded. In particular, for $(H,Y)\in \mathcal{Q}$ we must have indeed that $H$ is a normal subgroup of $\text{\text{Stab}}(Y)$, so that point (\ref{small cancellation theorem isom of staby}) of Lemma~\ref{small cancellation theorem} actually makes sense. Moreover, this also implies that, for all $\hat{v} \in \hat{v}(Q)$, if we write $Y$ for the subset of a pair $(H,Y) \in \mathcal{Q}$ such that $\hat{v}$ is the image of the vertex corresponding to $Z(Y)$, then $\text{\text{Stab}}(\hat{v}) \cong \text{\text{Stab}}(Y) / H$. Note that $\text{\text{Stab}}(\hat{v})$ is finite since $\text{\text{Stab}}(Y)$ is virtually cyclic infinite (by Lemma~\ref{lox subg in wpd is virt cyc}) and $H$ is infinite, and thus has finite index in $\text{\text{Stab}}(Y)$. Moreover, if $F$ is the maximal finite normal  subgroup of $\text{Stab}(Y)$, then its isomorphic image $\hat F$ in $\text{Stab}(\hat v)$ is normal in this subgroup. Furthermore, $\text{Stab}(\hat v) / \hat F$ is isomorphic either to $D_n$ or to $C_n$, and  the quotient map takes the elliptic elements of $\text{Stab}(Y)$  into either the trivial element or an involution of $D_n$.
    \end{remark}

We now state a number of properties of the space $\hat{X}$ and of the action of the group $\hat{G}$ on it that will prove to be useful further forward.

\begin{lemma}\cite[Proposition 5.8]{coulon_1}
\label{vQ has 2 elements}
    The set $\hat{v}(\mathcal{Q})$ has at least two elements.
\end{lemma}

\begin{lemma}\cite[Corollary 5.10]{coulon_1}
\label{image out stab y fixes uniq vert}
    Let $(H,Y) \in \mathcal{Q}$ and $v$ be the apex of $Z(Y)$. Let $\hat{g} \in \text{\text{Stab}}(\hat{v})$ be such that it is not the image of an elliptic element of $\text{\text{Stab}}(Y)$. Then, there exists $k \in \mathbb{Z}$ such that $A_{\hat{g}^{k}}$ is contained in the $6 \hat \delta$-neighbourhood of $\hat v$. In particular, $\hat{v}$ is the unique apex of $\hat{v}(\mathcal{Q})$ fixed by $\hat{g}$.
\end{lemma}

\begin{lemma}\cite[Proposition 3.21]{coulon_2}
\label{lifting figures}
    Let $\alpha \geq 0$ and $d \geq \alpha$. Let $\hat{Z}$ be an $\alpha$-quasi-convex subset of $\hat{X}$. We assume that, for every $\hat{v} \in \hat{v}(\mathcal{Q})$, $\hat{Z}$ does not intersect $B(\hat{v},\rho /20+d+10\hat{\delta})$. Then, there exists a subset $Z$ of $\dot X$ with the following properties.
    \begin{enumerate}
        \item The map $\zeta: \dot X \longrightarrow \hat{X}$ induces an isometry from $Z$ onto $\hat{Z}$.
        \item For every $\hat{g} \in \hat{G}$, if $\hat{g} \cdot \hat{Z}$ lies in the $d$-neighbourhood of $\hat{Z}$, then there exists a preimage $g$ of $\hat{g}$ such that for every $z,z' \in Z$, we have that \[ d_{\dot X}(g \cdot z',z)=d_{\hat{X}}(\hat{g} \cdot \hat{z}',\hat{z}). \]
        \item The projection $G \twoheadrightarrow \hat{G}$ induces an isomorphism from $\text{\text{Stab}}(Z)$ onto $\text{\text{Stab}}(\hat{Z})$.
    \end{enumerate}
\end{lemma}

The next result studies the structure of loxodromic subgroups in the quotient $\hat{G}$. The proof follows very closely the proof of \cite[Proposition 5.26]{coulon_1}.


\begin{lemma}\emph{(Compare \cite[Proposition 5.26]{coulon_1})}
\label{loxodromic subgroups in the quotient}
    Let $\hat E$ be a loxodromic subgroup of $\hat{G}$, $\hat{F}$ its maximal normal finite subgroup. Then one of the following holds.
    \begin{enumerate}
        \item There exists a loxodromic subgroup $E$ of $G$ such that the projection $G \twoheadrightarrow \hat{G}$ induces an isomorphism from $E$ onto $\hat E$ (in particular, it induces an isomorphism from the maximal normal finite subgroup $F$ of $E$ onto $\hat F$).
        \item There exists a pair $(H,Y)$ in $\mathcal{Q}$ with the following property: there is a finite normal subgroup $F'$ of $\text{Stab}(Y)$ such that its image is a subgroup of index at most $2$ in $\hat{F}$. Moreover, if $\text{Stab}(Y)$ is of cyclic type, then $\hat{F}$ coincides with the image of $F'$.
    \end{enumerate}

\end{lemma}

\begin{proof}
    Let $\hat{E}$ be a loxodromic subgroup of $\hat G$ and let $\hat{F}$ be the maximal normal finite subgroup of $\hat E$. Now, since $\hat{E}$ normalizes $\hat F$, then $C_{\hat F}$ is invariant under $\hat{E}$. Thus, $ \hat E \leq \text{Stab}(C_{\hat F})$. Moreover, by Lemma~\ref{charac subset of elliptic is quasi convex}, $C_{\hat F}$ is $9 \hat \delta$-quasi-convex.

    Suppose now that for every $\hat v \in \hat v (\mathcal{Q})$ we have that $C_{\hat F}$ does not intersect $B(\hat v , \rho /10)$. By the choice of $\rho$, we have that $\rho/10 > \rho/20 + 19 \hat \delta$, so we can apply Lemma~\ref{lifting figures} to get a subgroup $E$ of $G$ such that the projection $G \twoheadrightarrow \hat G$ induces an isomorphism from $E$ onto $\hat{E}$. Write $F$ for the preimage of $\hat F$ in $E$. Note that $E$ normalizes $F$. Furthermore, as the preimage of a loxodromic subgroup of $\hat G$, by Theorem~\ref{small cancellation theorem} (\ref{g hat preserves elem}) $E$ is loxodromic. Therefore, we are in the first case.

    Assume now there is some $\hat v \in \hat v (\mathcal{Q})$ such that $C_{\hat F}$ intersects $B(\hat v , \rho /10)$.  Let $\hat x$ be a point in this intersection. Then, for all $\hat u \in \hat F$, the triangle inequality gives \[ d(\hat v , \hat u \cdot \hat v) \leq d(\hat v , \hat x) + d(\hat x , \hat u \cdot \hat x) + d(\hat u \cdot \hat x, \hat u \cdot \hat v) \leq \rho /5 + 11 \hat \delta < 2 \rho. \]
    Thus, since $\hat G$ acts on $\hat v (\mathcal{Q})$ and the distance between distinct apices is at least $2 \rho$, then $\hat u$ must be in $\text{Stab}(\hat v)$. Let $(H,Y) \in \mathcal{Q}$ be a pair such that the image of the apex of its cone on $\hat X$ is $\hat v$.
    
    Suppose now that $\hat F$ contains an element $\hat u$ that is not the image of an elliptic element of $\text{Stab}(Y)$. Then, by Lemma~\ref{image out stab y fixes uniq vert}, a power $\hat{u}^{k}$ of this element has its axis contained in the $6 \hat \delta$-neighbourhood of $\hat v$. This implies that the characteristic subset $C_{\hat F}$ is contained in the $15 \hat \delta$-neighbourhood of $\hat v$: indeed, for $x\in C_{\hat F}$ we have $d(x,u^k(x))\leq 11\delta$ (by definition of $C_{\hat F}$), thus by Proposition 3.10 in \cite{coulon_1} the distance between $x$ and the axis of $u^k$ is at most $8.5\delta$; now, since the axis of $u^k$ is contained in the $6 \hat \delta$-neighbourhood of $\hat v$, we get $d(x,\hat v)\leq 14.5\delta$. By Lemma~\ref{charac subs of max finite sgrp}, the cylinder $Y_{\hat g}$ is contained in the $51 \hat \delta$-neighbourhood of the $C_{\hat F}$. But $Y_{\hat g}$ has bi-infinite local quasi-geodesics joining $\hat{g}^{\pm \infty}$, so it is not contained in a bounded set.

    Therefore, $\hat F$ cannot contain an element $\hat u$ that is not the image of an elliptic element of $\text{Stab}(Y)$. Now by Remark~\ref{stab of v in the small cancellation theorem} we are in the second case.
\end{proof}

The next result uses the techniques from \cite[Proposition 5.14]{coulon_1} for WPD actions on hyperbolic length spaces, and in \cite[Proposition 5.33]{dahmani_guirardel_osin}  for acylindrical actions on hyperbolic geodesic spaces.

\begin{lemma}
\label{acyl passes to quotient}
    Suppose that there exists $m \in \mathbb{N}$ such that $\lvert \text{Stab}(\hat{v}) \rvert \leq m$ for all $\hat{v} \in \hat{v}(\mathcal{Q})$. Then, the action of $\hat{G}$ on $\hat{X}$ is acylindrical.
\end{lemma}

\begin{proof}
    We will apply Remark~\ref{acylind characterization}. The action of $G$ on the cone-off $\dot X$ is acylindrical, therefore, there are positive numbers $M'$ and $L'$ such that, for all $x,x' \in \dot X$, if $d_{\dot X}(x,x')$ is at least $L'$, then there are at most $M'$ elements moving $x$ and $x'$ less than $100 \hat{\delta}$. We will show that we can take $\max \{ m,M' \}$ and $L'$ as the parameters $M$ and $L$ of Remark~\ref{acylind characterization}.

    Now, let $\hat{a},\hat{b} \in \hat{X}$ be at distance at least $L'$, and consider the set $\hat{Y}=\text{hull}(\{a,b\})$. By Lemma~\ref{hull is quasi-convex}, it is a $6\hat{\delta}$-quasi-convex subset of $\hat{X}$. Moreover, by Lemma~\ref{distances within quadrangles}, for every $\hat{x} \in \hat{Y}$, every element $\hat{g} \in \hat{G}$ moving $\hat{a}$ and $\hat{b}$ by at most $100 \hat{\delta}$ will move $\hat{x}$ by at most $600 \hat{\delta}$.

    Suppose that there is $\hat{v} \in \hat{v}(\mathcal{Q})$ lying in the closed $\rho/10$-neighbourhood of $\hat{Y}$, and let $\hat{c} \in \hat{Y}$ be at distance at most $\rho/10$ of $\hat{v}$. Then, we get from the triangle inequality that \[ d(\hat{g} \cdot \hat{v}, \hat{v}) \leq 2d(\hat{c}, \hat{v}) + d(\hat{g} \cdot \hat{c}, \hat{c}) \leq \rho/5+600 \hat{\delta}. \]
Recall that $\rho \geq \max  (\{ \rho_{0}, 10^{10}\boldsymbol{\delta},10^{20}L_{S}\boldsymbol{\delta}\})$ and $\hat{\delta}=64 \cdot 10^{4} \boldsymbol{\delta}$. Therefore, $\rho/5+600 \hat{\delta}$ (and thus also $d(\hat{g} \cdot \hat{v}, \hat{v})$) is less than $2\rho$, and this quantity is a lower bound for the distance between apices. In consequence, $\hat{g}$ fixes $\hat{v}$, and therefore, there can only be $m$ such elements.

    Now suppose that, for all $\hat{v} \in \hat{v}(\mathcal{Q})$, $\hat{Y}$ does not intersect $B(\hat{v},\rho/10)$. By the choice of $\rho$, we have that $\rho/20 + 610 \hat{\delta}<\rho/10$, thus we can apply Lemma~\ref{lifting figures} with $d=600 \hat{\delta}$ and $\alpha=6\hat{\delta}$: if some $\hat{g} \in \hat{G}$ moves $\hat{a}$ and $\hat{b}$ by at most $100 \hat{\delta}$, then $\hat{g} \cdot \hat{Y}$ lies in the closed $600 \hat{\delta}$-neighbourhood of $\hat{Y}$ (by Lemma~\ref{distances within quadrangles}), and by Lemma~\ref{lifting figures} we can find preimages $a$ and $b$ of $\hat{a}$ and $\hat{b}$ respectively, with \[d_{\dot X}(a,b)=d_{\hat{X}}(\hat{a},\hat{b}) \geq L',\] and such that there is a preimage $g$ of $\hat{g}$ in $G$ satisfying \[ d_{\dot{X}}(g \cdot a, a) =d_{\hat{X}}(\hat{g} \cdot \hat{a}, \hat{a}) \leq 100 \hat{\delta} \] and \[ d_{\dot{X}}(g \cdot b, b) =d_{\hat{X}}(\hat{g} \cdot \hat{b}, \hat{b}) \leq 100 \hat{\delta}. \] By the choice of $L'$, there can be at most $M'$ such elements $g$, yielding that there can be at most $M'$ elements $\hat{g}$ moving both $\hat{a}$ and $\hat{b}$ at most $100\hat{\delta}$.
    \end{proof}



We now want to find a bound for the injectivity radius of the images of certain subsets of $N$ (recall that $N$ is a normal subgroup of $G$ that has no involution, as defined at the beginning of Section 4.3). The following technical definition isolates the properties that a subset $Q$ of $N$ needs to have so that we can bound the injectivity radius of its image in terms of the injectivity radius of $Q$. In our applications in Section 5, we will take for $G$ a group in class $\classcprime$ (or in an auxiliary class $\classcprimezero$ that will be defined in that section), for $Q$ the set of translations, and for $N$ the subgroup of $G$ generated by translations.

\begin{definition}
\label{def stable family}
    Let $Q$ be a subset of $N$ and $\hat{Q}$ be its image in $\hat{G}$. We say that $Q$ is \emph{stable} with respect to $\mathcal{Q}$ if the following property is satisfied: let $\hat{g}$ be a non-elliptic element of $\hat{Q}$. Suppose there is a subset $A$ of $\dot{X}$ such that the projection $\zeta: \dot X \longrightarrow \hat{X}$ induces an isometry from $A$ onto the axis $A_{\hat{g}^{m}}$ for some $m \in \mathbb{N}$ and the projection $G \twoheadrightarrow \hat{G}$ induces an isomorphism from $\text{Stab}(A)$ onto $\text{Stab}(A_{\hat{g}^{m}})$. Let $g$ be the preimage of $\hat{g}$ in $\text{Stab}(A)$. Then $g \in Q$.
\end{definition}

We will prove in Lemma~\ref{tuple of cprime satisfies ind hyp} that in our setting the set of translations is stable in the sense of Definition~\ref{def stable family} (with respect to an appropriate family $\mathcal{Q}$). 


Let us remark that being a stable subset may depend on the specific family $\mathcal{Q}$ we are considering. However, if the family $\mathcal{Q}$ is clear by context, we may omit mentioning it explicitly. Notice also that a subgroup of $N$ containing $K$ will be a stable subset. In particular, $N$ is a stable subset independently of the family in consideration. The next lemma slightly generalizes \cite[Proposition 5.31]{coulon_1}.

\begin{lemma}
\label{asymp trans length stable subset}\emph{(Compare \cite[Proposition 5.31]{coulon_1})}
    Let $Q$ be a stable subset of $N$. Denote by $l$ the infimum over the asymptotic translation length in $X$ of loxodromic elements of $Q$ that do not belong to $\text{Stab}(Y)$ for $(H,Y) \in \mathcal{Q}$. Let $\hat{g}$ be a non-elliptic element of $\hat{Q}$. If every preimage of $\hat{g}$ in $G$ is loxodromic, then we have \[[\hat{g}]^{\infty} \geq \text{min} \biggl ( \biggl \{ \frac{l \hat{\delta}}{\pi \text{sinh}(26 \hat{\delta})} , \hat{\delta} \biggr \} \biggr ).\]
\end{lemma}

\begin{proof}
    As stated before, this lemma appears in \cite{coulon_1} as Proposition 5.31 taking the stable subset $Q$ to be the whole subgroup $N$. However, the proof works with no further change if we put $Q$ in place of $N$, since we have included \textit{ad hoc} in Definition~\ref{def stable family} the property that we need for this to happen.
\end{proof}

From this result we immediately obtain the following corollary and the desired bound.

\begin{corollary}
\label{inj radius of stable subset}
    Let $Q$ be a stable subset of $N$. Denote by $l$ the infimum over the asymptotic translation length in $X$ of loxodromic elements of $Q$ that do not belong to $\text{Stab}(Y)$ for $(H,Y) \in \mathcal{Q}$. Then we have \[ r_{\text{inj}}(\hat{Q}, \hat{X}) \geq \text{min} \biggl ( \biggl \{ \frac{l \hat{\delta}}{\pi \text{sinh}(26 \hat{\delta})} , \hat{\delta} \biggr \} \biggr ).\]
\end{corollary}


\subsection{Small Cancellation and Partial Periodic Quotients}
\label{subsubsection partial periodic and small cancellation}


\subsubsection{Small cancellation quotients (SC-quotients)}

\begin{proposition}\emph{(Compare \cite[Proposition 6.1]{coulon_1})}
\label{res: SC - induction lemma}\label{rinj in the sc quotient} \label{acyl passes to sc quotient}
	There exist positive constants $\rho_0$,  $\delta_1$ and $L_S$ such that for every integer $\nu_0$  there is an integer $n_0$ such that the following holds:
 
	Suppose $G$ is a group acting by isometries on a $\delta_1$-hyperbolic length space $X$ such 
	that this action is WPD and non-elementary.
	Let $N$ be a normal subgroup of $G$ without involutions, $Q$  a conjugation invariant set of elements of $N$, and let $P$ be the subset of loxodromic elements $h$ of $Q$ which are primitive as elements of $N$ and such that $ [h] \leq L_S\delta_1 $.
Let $n_1 \geq n_0$ and $n\geq n_1$ be odd and suppose furthermore that:
 \begin{enumerate}
		\item \label{enu: SC - induction lemma - e}
		$e(N,X)$ divides $n$,
		\item \label{enu: SC - induction lemma - nu}
		$\nu(N,X) \leq \nu_0$,
		\item \label{enu: SC - induction lemma - A}
		$A(N,X) \leq   6\pi \nu_0 \sinh (2L_S\delta_1)$,
		\item \label{enu: SC - induction lemma - rinj}
		$ r_{\text{inj}} (Q,X) \geq 2\delta_1 \sqrt {\frac {L_S\sinh (\rho_0)}{n_1\sinh (26 \delta_1)}}$.
	\end{enumerate}

	 Let $K=\langle h^n \, : \, h \in P\rangle^G$, put $\hat G=G/K$ and let $\hat{N}, \hat{Q}$ denote the images of $N$ and $Q$ under the quotient map. Then $\hat{N}$ contains no involution and  there exists a $\delta_1$-hyperbolic length space $\hat X$ such that the tuple $(\hat G , \hat N , \hat Q , \hat X)$  satisfies the conditions (\ref{enu: SC - induction lemma - e}) -- (\ref{enu: SC - induction lemma - A})  of this proposition  with $\hat{N} $ and $ \hat{X}$ in place of $N $ and $ X$ respectively.  
  
Additionally, if $Q$ is stable with respect to $\mathcal{Q}=\{(\langle h^{n} \rangle, Y_{h}) \, : \, h \in P\}$, then the tuple $(\hat G, \hat N, \hat Q, \hat X)$ satisfies condition (\ref{enu: SC - induction lemma - rinj}) with $\hat Q$ and $\hat X$ in place of $Q$ and $X$ respectively.

	Moreover, the quotient map $G \rightarrow \hat G, g\mapsto \hat{g}$ has the following properties.
	\begin{itemize}
		\item  For  $g \in G$,  we have 
		\begin{displaymath}
			[{\hat g}]^{\infty}_{\hat X} \leq \frac 1{\sqrt {n_1}} \left(\frac {4\pi}{\delta_1}\sqrt{\frac {\sinh (\rho_0)\sinh (26 \delta_1)}{L_S}}\right)[g]^{\infty}_{ X}. 
		\end{displaymath}
		\item For every non-loxodromic elementary subgroup $E$ of $G$, the quotient map $G \rightarrow \hat G$ induces an isomorphism from $E$ onto its image $\hat E$ which is elementary and non-loxodromic.
		\item Let $\hat g$ be an elliptic (respectively parabolic) element of $\hat N$. 
		Either $\hat g^n = 1$ or $\hat g$ is the image of an elliptic (respectively parabolic) element of $N$.
		\item Let $u,u' \in N$ be such that $ [u] < L_S \delta_1$ and $u'$ is elliptic. 
		If $\hat{u}$ and $\hat{u}'$ are conjugate in $\hat G$ then so are $u$ and $u'$ in $G$.
	\end{itemize}

Furthermore, if the action of $G$ on $X$ is acylindrical, and if there is some positive integer $m$ such that for all $h \in P$ we have $\lvert \text{\text{Stab}}(Y_{h})/\langle h^{n} \rangle \rvert \leq m$, then the action of $\hat{G}$ on $\hat{X}$ is also acylindrical.


\end{proposition}

\begin{proof}
First of all, note that Proposition 6.1 in \cite{coulon_1} involves three constants $A_0,r_0,\alpha$ that do not appear in Proposition~\ref{res: SC - induction lemma} above, while Proposition~\ref{res: SC - induction lemma} involves a constant $\rho_0$ that does not appear in \cite[Proposition 6.1]{coulon_1}. This difference is purely cosmetic, since the constants $A_0,r_0,\alpha$ can be expressed explicitly in terms of $L_S,\delta_1,\rho_0$ (and we refer the reader to the very beginning of the proof of Proposition 6.1 in \cite{coulon_1} for the explicit formulas), and the assumptions (\ref{enu: SC - induction lemma - e}) to (\ref{enu: SC - induction lemma - rinj}) above are simply a rewriting of the assumptions (1) to (4) in \cite[Proposition 6.1]{coulon_1} in terms of $\rho_0$ instead of $A_0,r_0,\alpha$, and with $Q$ in place of $N$ in the case of assumption (\ref{enu: SC - induction lemma - rinj}).

With this in mind the proof of Proposition 6.1 in \cite{coulon_1} works for Proposition~\ref{res: SC - induction lemma} modulo some mostly minor changes. We include some details of the proof, focusing on the construction of the space $\hat{X}$, and we refer to  the proof of Proposition 6.1 in \cite{coulon_1} whenever parts of this proof work exactly as in that case.

Fix a positive integer $\nu_0$ . We will define a rescaling parameter $\lambda$ for the space $X$. For a positive integer $k$, we put \[\lambda_{k}=\frac{4\pi}{\delta_{1}} \sqrt{\frac{\sinh(\rho_{0}) \sinh(26\delta_{1})}{kL_{S}}} .\]
Now, we set the critical exponent $n_0$ as the smallest integer greater than 100 such that for every integer $k \geq n_0$ we have \[ \lambda_k \delta_1 \leq \delta_0, \] \[ \lambda_k (6\pi \nu_0 \sinh(2L_S \delta_1)+118 \delta_1) \leq \min \{ \Delta_0 \, , \, \pi \sinh (2L_S \delta_1)\},\] \[ \lambda_k \frac{L_S \delta_{1}^{2}}{2 \pi \sinh (26 \delta_1)} \leq \delta_1, \] and \[ \lambda_k \rho_0 \leq \rho_0, \] where the parameters $\rho_0$, $\delta_0$ and $\Delta_0$ are the ones appearing in Lemma~\ref{small cancellation theorem}. Let $n_1 \geq n_0$ and $n \geq n_1$ be an odd integer, and write $\lambda = \lambda_{n_1}$.

Let $G$ be a group with a normal subgroup $N$ acting on a hyperbolic space $X$, and let $Q$ be a family of conjugation invariant elements of $N$. Assume that $(G,N,Q,X)$ satisfies the assumptions of Proposition~\ref{res: SC - induction lemma} for $\nu_0$ and $n$. Remember that we denote by $\hat{G}$ the quotient of $G$ by the normal subgroup $K$ generated by $\{ h^n \, : \, h \in P \}$, where $P$ is as in the statement of the proposition. Since $P \subseteq N$, $K$ is a (normal) subgroup of $N$, and we write $\hat{N}= N / K$.

We consider for the rest of this proof the action of $G$ on the rescaled space $\lambda X$ (unless explicitly stated otherwise). This is a $\delta$-hyperbolic space with $\delta= \lambda \delta_1 \leq \delta_0$. We denote by $\mathcal{Q}$ the family $\mathcal{Q}=\{(\langle h^{n} \rangle, Y_{h}) \, : \, h \in P\}$. We claim that this family satisfies $\Delta(\mathcal{Q}) \leq \Delta_0$ and $T(\mathcal{Q}) \geq 8 \pi \sinh(\rho_0)$. The proof of the first inequality works exactly as in Lemma 6.2 in \cite{coulon_1}. For the second inequality, notice that by Remark~\ref{invariants in rescaled space}, assumption~(\ref{enu: SC - induction lemma - rinj}) and the choice of $n$ we have \[ r_{\text{inj}}(Q, \lambda X) \geq \frac{8 \pi \sinh(\rho_0)}{n_1} \geq  \frac{8 \pi \sinh(\rho_0)}{n}.\] Thus, since $P$ is a subset of $Q$, for all $h \in P$ we get $[h^{n}]^{\infty} =n [h]^{\infty} \geq 8 \pi \sinh(\rho_0)$, from which the desired inequality follows.

On account of this, we can apply Lemma~\ref{small cancellation theorem} to the action of $G$ on $\lambda X$ and the family $\mathcal{Q}$, so we denote by $\dot{X}$ the cone-off or radius $\rho_0$ over $X$ relative to the family $\{ Y \, : \, (H,Y) \in \mathcal{Q}\}$, and we set $\hat{X}$ as the quotient of $\dot{X}$ by $K$.

The proof that $\hat{G}$, $\hat{N}$ and $\hat{X}$ satisfy the required properties now follows exactly as in the proof of Proposition 6.1 in \cite{coulon_1}. 

In particular, the results from Lemma~\ref{small cancellation theorem} to Lemma~\ref{inj radius of stable subset} apply to the action of $\hat G$ on $\hat X$.

Suppose now that  $G$ acts acylindrically on $X$.  Then by Proposition~\ref{res: SC - induction lemma}, and the facts that acylindricity is preserved by rescalings and that, by Lemmas~\ref{acyl passes to cone-off} and~\ref{acyl passes to quotient}, the acylindricity of the action is also preserved by the cone-off construction (for $G$) and by the SC-quotient (for $\hat G$, if there exists a bound for $\lvert \text{\text{Stab}}(Y_{h})/\langle h^{n} \rangle \rvert$) respectively.

Finally, assume that $Q$ is stable with respect to $\mathcal{Q} $. If $g$ is a loxodromic isometry of $Q$ that does not belong to $\text{Stab}(Y)$ for any $(H,Y) \in \mathcal{Q}$, then by Lemma~\ref{asympt and trans length relation}, its asymptotic translation length satisfies \[ [g]^{\infty} \geq \lambda (L_S \delta_1 - 32 \delta_1).\] The choice of $L_S$ now implies $[g]^{\infty} \geq L_S \delta_1 /2$ and we obtained the required estimate from Corollary~\ref{inj radius of stable subset} and the choice of $n_1$.
\end{proof}

\textbf{Terminology.} For the remainder of the article, we call a group obtained as the quotient of a group $G$ as provided by Proposition~\ref{res: SC - induction lemma} a \emph{small cancellation quotient} (or simply an \emph{SC-quotient}) of $G$, and we will denote it by $\hat{G}$. Similarly, if $X$ is the length space on which $G$ acts, we will write $\hat{X}$ for the length space on which $\hat{G}$ acts.

\medskip

Notice that by Proposition~\ref{res: SC - induction lemma}, the quotient group $\hat G$ and the hyperbolic space $\hat X$ on which it acts exist whether the family $Q$ is stable with respect to $\mathcal{Q}$ or not. However, we want to apply Proposition \ref{res: SC - induction lemma} iteratively when proving Theorem \ref{res : SC - partial periodic quotient}, and for this purpose we will need to take a quotient of $\hat G$. Therefore, we need to bound the injectivity radius of $\hat Q$, and to this purpose is that we need to assume the stability of $Q$. In particular, we will have that if $\hat Q$ is stable with respect to $\hat{\mathcal{Q}}$ (where this family is constructed from $\hat Q$ exactly as $\mathcal{Q}$ was constructed from $Q$), then $(\hat G, \hat N, \hat Q, \hat X)$ satisfies all assumptions of Proposition \ref{res: SC - induction lemma}. This observation also motivates the following strengthening of Definition~\ref{def stable family} for stable families.

\begin{definition}
\label{def strongly stable family}
    Let $Q$ be a subset of $N$. We say that $Q$ is \emph{strongly stable} if the following property holds.
    
    Suppose $\{ \hat{G}_{i} \, : \, 0 \leq i \leq k \}$ is a finite sequence of quotients obtained from $G=\hat{G}_0$ by successive applications of Lemma~\ref{small cancellation theorem}, where the space on which $\hat{G}_0$ acts is (a possibly rescaled version of) $X$, and the space $\hat{X}_{i+1}$ on which $\hat{G}_{i+1}$ acts is (a possibly rescaled version of) the space $\hat{X}$ provided by Lemma~\ref{small cancellation theorem} when we consider the action of $\hat{G}_i$ on $\hat{X}_i$. Suppose furthermore that at step $i$ the family $\mathcal{Q}_i$ was built by taking the subsets $H_i$ of $\hat{G}_i$ to be the $n$-th power of a loxodromic element $\hat h$ in the image of $Q$ in $\hat{G}_i$ that is primitive as an element of the image $\hat{N}_i$ of $N$ in $\hat{G}_i$, and $Y$ is the cylinder of this $n$-th power.

    Then  the image $\hat{Q}_i$ of $Q$ in $\hat{G}_i$ is stable with respect to $\mathcal{Q}_i$.
\end{definition}

\begin{corollary}\label{cor: SC - induction lemma}
If the tuple $(G, N, Q, X)$ satisfies the assumptions of Proposition~\ref{res: SC - induction lemma} and $Q$ is strongly stable, then  so does   $(\hat G , \hat N , \hat Q, \hat X)$, and $\hat Q$ is strongly stable as well.
\end{corollary}

Thus, under the assumption that $Q$ is strongly stable, we can indeed iterate the application of Proposition~\ref{res: SC - induction lemma}.

\subsubsection{Partial periodic quotients (PP-quotients)}
\label{subsubsection pp quotients}


\begin{theorem}\emph{(Compare \cite[Theorem 6.9]{coulon_1})}
\label{res : SC - partial periodic quotient}
	Let $X$ be a $\delta$-hyperbolic length space.
	Let $G$ be a group acting by isometries on $X$.
	We suppose that this action is WPD and non-elementary.
	Let $N$ be a normal subgroup of $G$ without involutions, and $Q$ be a conjugation invariant strongly stable subset of $N$.
	In addition we assume that $e(N,X)$ is odd, $\nu(N,X)$ and $A(N,X)$ are finite and $ r_{\text{inj}}(Q,X)$ is positive.
	There exist a normal subgroup $K$ of $G$ contained in $N$ and a critical exponent $n_1$ that only depends on $\delta, \nu (N,X), A(N,X)$ and $r_{inj}(Q,X)$ such that for every odd integer $n \geq n_1$ which is a multiple of $e(N,X)$ the following holds:
	\begin{enumerate}
		\item if $E$ is an elementary subgroup of $G$ which is not loxodromic, then the projection $G \twoheadrightarrow \bar{G} = G/K$ induces an isomorphism from $E$ onto its image;
        \label{pp quotient isomorph of elliptic on projection}
		\item every non-trivial element of $K$ is loxodromic;
        \label{pp quotient K is loxo}
		\item for every element $\bar{g}$ of finite order in the image $\bar{N}$ of $N$ in $\bar{G}$, either $\bar{g}^n=1$ or $\bar{g}$ is the image of an elliptic element of $N$. Moreover, for every element $h$ in $Q$, either its image $\bar{h}$ in $\bar{N}$ satisfies $\bar{h}^n=1$ or it is identified with the image of a non-loxodromic element of $N$;
        \label{pp quotient image of N is partial periodic}
		\item there are infinitely many elements in $\bar{N}$ which do not belong to the image of an elementary non-loxodromic subgroup of $G$.
        \label{pp quotient infinitely many elem in bar N}
	\end{enumerate}
\end{theorem}

Notice that the third consequence of Theorem~\ref{res : SC - partial periodic quotient} implies that $\bar{N}$ contains no involution since $N$ contains no involution and $n$ is odd.


\begin{proof}[Proof of Theorem~\ref{res : SC - partial periodic quotient}]
We will obtain Theorem~\ref{res : SC - partial periodic quotient} by iterated applications of Proposition~\ref{res: SC - induction lemma}. More specifically, by putting $G=G_{0}$, and  building a sequence $(G_{i})_{i \in \mathbb{N}}$ where we get $G_{i+1}$ from $G_{i}$ by adding new relations of the form $h^{n}$ with $h$ a loxodromic element of the image of $Q$ in $G_{i}$ that is primitive as an element of the image of $N$ in this quotient. The group $\bar{G}$ in Theorem~\ref{res : SC - partial periodic quotient} is defined as the limit of this sequence. 

We now give the details of the construction  and refer the reader to the proof of Theorem 6.9 in \cite{coulon_1} whenever parts of it work exactly as in that case.

Keeping the same notation for the constants $L_S$, $\rho_0$ and $\delta_1$ as in Proposition~\ref{res: SC - induction lemma}, we  let $X_0= \lambda' X$ be a rescaled version of $X$ where  the rescaling constant $\lambda'$ is the greatest real number such that $\delta'\leq \delta_1$ and $A(N,X_0) \leq   6\pi \nu_0 \sinh (2L_S\delta_1)$.   Thus, in $X_0$ the hyperbolicity constant is $\delta'= \lambda' \delta$, while the invariants satisfy $e(N,X_0)=e(N,X)$, $\nu(N,X_0)=\nu(N,X)$, $r_{\text{inj}}(Q,X_0)= \lambda' r_{\text{inj}}(Q,X)$ and $A(N,X_0)=\lambda' A(N,X)$. We write $\nu_0=\nu(N,X)$. 
The critical exponent $n_1$ is defined as the smallest positive integer such that \[  r_{\text{inj}} (Q,X_0) \geq 2\delta_1 \sqrt {\frac {L_S\sinh (\rho_0)}{n_1\sinh (26 \delta_1)}} \]and \[1 < \frac 1{\sqrt {n_1}} \left(\frac {4\pi}{\delta_1}\sqrt{\frac {\sinh (\rho_0)\sinh (26 \delta_1)}{L_S}}\right).\] 
We will write $c_1$ for the constant appearing in the first of these equations and $c_2$ for the constant appearing in the second one. Notice that indeed the value of $n_1$ only depends on $\delta, \nu (N,X), A(N,X)$ and $r_{inj}(Q,X)$. Fix  an odd integer $n \geq n_1$ that is a multiple of $e(N,X)$. Denote by $P_0$ the (conjugation invariant) set of elements $h$ in $Q_0$ that are loxodromic, primitive as elements of $N_0$ and that satisfy $ [h] \leq L_S\delta_1 $. Set $\mathcal{Q}_0= \{ (\langle h^{n} \rangle, Y_{h^{n}}) \, : \, h \in P_0 \}$. Since $Q$ is a strongly stable subset of $N$, the set $Q_0$ is again strongly stable with respect to $\mathcal{Q}_0$. Thus, by construction, and with $N_0=N$ and $Q_0=Q$, the tuple $(G_0,N_0,Q_0,X_0)$ satisfies the assumptions of Proposition~\ref{res: SC - induction lemma} for $\nu_0$, exponent $n$ and the family $\mathcal{Q}_0$.

Now assume that we have constructed a group $G_i$ acting on a hyperbolic space $X_i$ in such a way that the tuple $(G_i,N_i,Q_i,X_i)$ satisfies the assumptions of Proposition~\ref{res: SC - induction lemma} for $\nu_0$ and exponent $n$ (except that we do not require \textit{a priori} $Q_i$ to be stable with respect to some family), where $N_i$ and $Q_i$ are the images of $N_0$ and $Q_0$ respectively in $G_i$. Suppose that furthermore that for all $0< j \leq i$, the group $G_j$ has been obtained as a quotient of $G_{j-1}$ by a subgroup  $K'=\langle h'^n\rangle^{G_{j-1}}$ where $h$ is a loxodromic element of the image of $Q$ in $G_{j-1}$ that are primitive as elements of $N_{j-1}$. Denote by $P_i$ the (conjugation invariant) set of elements $h$ in $Q_i$ that are loxodromic, primitive as elements of $N_i$ and that satisfy $ [h] \leq L_S\delta_1 $. Put $\mathcal{Q}_i= \{ (\langle h^{n} \rangle, Y_{h^{n}}) \, : \, h \in P_i \}$. Since $Q$ is a strongly stable subset, $Q_i$ is again strongly stable with respect to $\mathcal{Q}_i$, so indeed the tuple $(G_i,N_i,Q_i,X_i)$ satisfies the assumptions of Proposition~\ref{rinj in the sc quotient} for $\nu_0$, exponent $n$ and $\mathcal{Q}_i$. Let $K_i=\langle h^n \, : \, h \in P_i\rangle^{G_i}$. By Proposition~\ref{res: SC - induction lemma}, the quotient $G_{i+1}=G_i / K_i$ acts on a hyperbolic space $X_{i+1}$, and if we write $N_{i+1}=N_i/K_i$ and $Q_{i+1}$ for the image of $Q$ in $G_{i+1}$, then the tuple $(G_{i+1},N_{i+1}, Q_{i+1}, X_{i+1})$  satisfies the assumptions of Proposition~\ref{res: SC - induction lemma} for $\nu_0$ and exponent $n$ (except that similarly we do not require \textit{a priori} $Q_{i+1}$ to be stable with respect to some family). Thus, we get a well-defined sequence $(G_i)_{i \in \mathbb{N}}$.

Let $\bar{G}$ be the limit of the sequence $(G_{i})_{i \in \mathbb{N}}$, so $\bar{G}$ is a quotient of $G$ by a normal subgroup $K$ contained in $N$. We claim that this group satisfies the required properties. The proofs of Properties (\ref{pp quotient isomorph of elliptic on projection}) and (\ref{pp quotient K is loxo}) work exactly as in the case of Theorem 6.9 in \cite{coulon_1}. For Property (\ref{pp quotient infinitely many elem in bar N}), it is enough to notice that this claim holds even for the set of all loxodromic elements of $N$, as is the case in Theorem 6.9 in \cite{coulon_1}.

It only remains to prove Property (\ref{pp quotient image of N is partial periodic}). A simple inductive argument using Proposition~\ref{res: SC - induction lemma} shows that if $g'$ is an elliptic (respectively parabolic) element of $N_i$, then either $g'^{n}=1$ or $g'$ is the image of an elliptic (respectively parabolic) element of $N$. Let now $\bar{g}$ be an element of finite order of $\bar{N}$ which is not the image of an elliptic element of $N$. Denote by $g$ a preimage of this element in $N$, and denote by $g_i$ the image of $g$ in $N_i$. Notice that $g$ must be loxodromic, since otherwise it would be a parabolic element (thus of infinite order) and the projection would induce an isomorphism of $\langle g \rangle$ onto its image $\langle \bar{g} \rangle$. If $g_i$ was of infinite order for all $i \in \mathbb{N}$, then $\bar{g}$ would be of infinite order, so there exists $j \in \mathbb{N}$ such that $g_j$ is of finite order. Now the first claim of Property (\ref{pp quotient image of N is partial periodic}) follows from the second and third consequences of Proposition \ref{res: SC - induction lemma}.

For the second part of this property, assume moreover that, under the notation of the previous paragraph, $\bar{g}$ is in $\bar{Q}$. By the way  the sequence $(G_{i})_{i \in \mathbb{N}}$ was constructed, we have  $[g_i] \leq (c_2)^{i}[g]$. Therefore, there exists $j \in \mathbb{N}$ such that $[g_j] < c_1=r_{\text{inj}}(Q_j,X_j)$. In particular, since $g_j \in Q_j$, it is not loxodromic. It also cannot be parabolic (since once again, this would give that $\bar g$ is of infinite order), so it must be elliptic, and once again the desired conclusion follows from the second and third consequences of Proposition \ref{res: SC - induction lemma}.
\end{proof}

\textbf{Terminology.} For the remainder of the article, we will call a group obtained as the quotient of a group $G$ as provided by Theorem~\ref{res : SC - partial periodic quotient} a \emph{partial periodic quotient} (or simply a \emph{PP-quotient}) of $G$, and we will denote it by $\bar{G}$.

\section[SC- and PP- Quotients]{Proof of Proposition~\ref{prop: classes stable under pp quotients}}
\label{section small cancellation quotient}

The main purpose of this section is to prove Proposition~\ref{prop: classes stable under pp quotients}, which we recall here.

\begin{proposition2.15}
\label{prop: classes stable under pp quotients restatement}
    Let $G$ be a group acting on a tree $X$ in such a way that the pair $(G, X)$  belongs to $\classcprime$. Then $G$ has a quotient $\bar G\in\classc$ such that the image of every pair of distinct involutions of $p$-affine type (respectively, of $p$-minimal type) is again 
    of $p$-affine type (respectively, of $p$-minimal type), any subgroup of $G$ that is elliptic for the action of $G$ on $X$ projects to an isomorphic image in $\bar{G}$ and $O(G)=O(\bar{G})$.
\end{proposition2.15}

In order to prove Proposition \ref{prop: classes stable under pp quotients}, we will apply Theorem~\ref{res : SC - partial periodic quotient} with $N=\langle Tr(G)\rangle$ and $Q=Tr(G)$. The reason why we don't simply take $Q=N$ as in \cite{coulon_1} is illustrated by the following example: if $i,j,k$ are three (pairwise distinct) involutions of $G$, then $(ijk)^2$ belongs to $N$ and it may happen that $(ijk)^2$ is primitive in $N$ (note that $ijk$ does not belong to $N$, otherwise $N$ would contain an involution), but we don't want to add the relation $((ijk)^2)^p=1$ in our quotient because this would create a new involution, namely $(ijk)^p$, over which we have no control (for instance, such an involution would have no reason to be conjugate to the other involutions), and thus this would break the conditions defining our classes of groups.

In order to apply Theorem~\ref{res : SC - partial periodic quotient} we need to show that $Q$ is strongly stable, and to do so we need to consider the intermediate small cancellation quotients obtained using Proposition~\ref{res: SC - induction lemma}. However, a pair $(\hat G , \hat X)$ obtained from a pair $(G,X)$ in $\classcprime$ does not belong to $\classcprime$ (because $\hat X$ is not a tree). For this reason we introduce a new auxiliary class $\classcprimezero$ containing $\classcprime$, and we will prove that this new class is stable under the small cancellation quotients obtained using Proposition~\ref{res: SC - induction lemma}.

\begin{definition}[Class $\classcprimezero$]
\label{definition class Cprimezero}
Let $G'$ be a group acting by isometries on a $\delta_1$-hyperbolic length space $X'$ (with $\delta_1$ the constant from Proposition~\ref{res: SC - induction lemma}). Write $N'= \langle Tr(G') \rangle$ and $Q'=Tr(G')$. We say that $(G', X')$ belongs to the class $\classcprimezero$ if it satisfies the conditions of the class $\classcprime$ where Condition~\ref{def class Cprime satisfies the ind lemma} is replaced by Condition~\ref{def class Cprimezero satisfies the ind lemma}:
  \begin{enumerate}[label=(\arabic*'')]
  \setcounter{enumi}{1}
      \item \label{def class Cprimezero satisfies the ind lemma}
         The tuple $(G',N',Q',X')$  satisfies the assumptions of Proposition~\ref{res: SC - induction lemma} for $\nu_0=5$ and exponent~$p$, except that $Q'$ is not required \textit{a priori} to be stable with respect to $\mathcal{Q}'$.
\end{enumerate}
\end{definition}

\begin{remark}
    Notice that if we want to obtain a PP-quotient of a group in $\classcprime$ by iteratively obtaining SC-quotients of groups in $\classcprimezero$, then we will need the family $Q'$ to be stable with respect to $\mathcal{Q}'$ (where this family is obtained from $Q'$ exactly as $\mathcal{Q}$ is obtained from $Q$ in Proposition \ref{res: SC - induction lemma}). However, we will see that this property follows directly from the other conditions defining $\classcprimezero$, so we do not need to include it explicitly in the definition above.
\end{remark}



At the end of this section we will also obtain Lemma \ref{non-commuting translations}, that implies in particular that the sharply 2-transitive groups constructed in Theorem \ref{embedding theorem} do not split.

\subsection[SC-Quotients]{Stability of the Class $\classcprimezero$ under SC-Quotients}
\label{subsec sc quotient}


In this subsection we will show that the new class $\classcprimezero$ defined above is stable under SC-quotients. We fix a group $G'$ acting on a hyperbolic space $X'$ in such a way that the pair $(G',X')$ belongs to $\classcprimezero$, and we write $N'= \langle Tr(G') \rangle$ and $Q'=Tr(G')$. The following result is the crucial induction step for the proof of Proposition~\ref{prop: classes stable under pp quotients}.

\begin{proposition}
\label{reformulation class cprime stable under sc quotients}
If  $(G',X')$ is in class $\classcprimezero$, then so is its
small cancellation quotient $(\hat G', \hat X')$ given by Proposition~\ref{res: SC - induction lemma}. Moreover, the image of every pair of distinct involutions of $p$-affine type (respectively, of $p$-minimal type) is again of $p$-affine type (respectively, of $p$-minimal type), and $O(G')=O(\hat G')$.
\end{proposition}

We need some preliminary lemmas. We first give a classification of the loxodromic subgroups of $G'$.

\begin{lemma}
\label{loxodromic subgroups in stprimezero}
    Let $E$ be a loxodromic subgroup of $G'$. Then, $E$ is isomorphic to one of the following four groups: $\mathbb{Z}$, $C_{2}\times \mathbb{Z}$, $D_{\infty}$ or $C_4\ast_{C_2}C_4$.
\end{lemma}

\begin{proof}
    By Lemma~\ref{lox subg in wpd is virt cyc}, we know that $E$ is virtually cyclic and thus there exists a (necessarily normal) finite subgroup $C$ of $E$ such that $E\simeq C\rtimes \mathbb{Z}$ or $E\simeq A\ast_C B$ with $[A:C]=[B:C]=2$. But $C$ has order at most 2 by Condition~\ref{def class Cprime no subg of odd ord norm by lox} of Definition~\ref{definition class Cprimezero}. Moreover, by Condition~\ref{def class C pairs of minimal or affine type} of Definition~\ref{definition anss2t}, no two involutions of $G'$ commute, which leads to the four possibilities listed above.
\end{proof}

Note that among these four groups, only $D_{\infty}$ contains two distinct involutions.

\begin{lemma}
\label{apex stab in SC quot}
    Let $\hat{v} \in \hat{v}(\mathcal{Q})$ be an apex of $\hat{X}'$. Then $\text{Stab}(\hat{v})$ is isomorphic to $D_{p}$.
\end{lemma}

\begin{proof}
Note that if $h=ij\in N'$ is a loxodromic translation that is primitive in $N'$ (where $i,j$ are two distinct involutions), then the setwise stabilizer $S$ of the pair $\{ h^{\pm \infty} \}\subset \partial G'$ contains two distinct involutions (namely $i$ and $j$), thus by Lemma~\ref{loxodromic subgroups in stprimezero} and the subsequent remark it is isomorphic to $D_{\infty}$. It follows that $S=\langle i,j\rangle$ since $h$ is primitive in $N'$ (and as a consequence, $h$ is also primitive in $G'$).
Now, let $\hat{v} \in \hat{v}(\mathcal{Q})$ be an apex of $\hat{X}'$. By the construction of the cone-off (see Section 4.3), and by our choice of the set $Q'$, there exists a primitive loxodromic translation $h=ij\in N'$ (where $i,j$ are two distinct involutions) such that $(\langle h^{p} \rangle, Y_{h})$ belongs to $\mathcal{Q}$ and $\text{Stab}(\hat v)$ is the quotient of the setwise stabilizer of the pair $\{ h^{\pm \infty} \}\in \partial G'$ by the normal subgroup generated by $h^p$. It follows from the previous paragraph that $\text{Stab}(\{ h^{\pm \infty} \})=\langle i,j \ \vert \ i^2=j^2=1\rangle$. Therefore, $\text{Stab}(\hat{v})=\langle i,j \ \vert \ i^2=j^2=(ij)^p=1\rangle\simeq D_{p}$.
\end{proof}

The next lemma follows easily from the previous lemma and from item (\ref{elliptic subgroups of hat g}) of Lemma~\ref{small cancellation theorem}.

\begin{lemma}
\label{elliptic maps to Dp}
    Let $\hat F$ be an elliptic subgroup of $\hat G'$. Then one of the following holds:
    \begin{enumerate}
        \item there exists an elliptic subgroup $F$ of $G'$ such that the quotient map $G'  \twoheadrightarrow \hat G'$ induces an isomorphism from $F$ onto $\hat F$,
        \item $\hat F$ is isomorphic to $C_p$ or $D_p$.
    \end{enumerate}
\end{lemma}

\begin{proof}
As a consequence of item (\ref{elliptic subgroups of hat g}) of Lemma~\ref{small cancellation theorem}, we only need to take care of the case when $\hat F$ is contained in $\text{Stab}(\hat v)$ for some $\hat v \in \hat v (\mathcal{Q})$. By Lemma~\ref{apex stab in SC quot}, $\text{Stab}(\hat v)$ is isomorphic to $D_p$, thus $\hat F$ is either trivial or isomorphic to $C_2$ or $C_p$ or $D_p$. More precisely, there is a primitive loxodromic translation $h\in N$ such that $\text{Stab}(\hat v)$ is the quotient of the setwise stabilizer $S$ of the pair $\{ h^{\pm \infty} \}\in \partial G'$ by the normal subgroup generated by $h^p$. By Lemma~\ref{loxodromic subgroups in stprimezero} $S$ is an infinite dihedral group. We need to prove that if $\hat F$ is trivial or isomorphic to $C_2$, then there exists an elliptic subgroup $F$ of $G'$ such that the projection $G'  \twoheadrightarrow \hat G'$ induces an isomorphism from $F$ onto $\hat F$. This is obvious if $\hat F$ is trivial, so let us assume that $\hat F$ is isomorphic to $C_2$, and write $\hat F=\langle i\rangle$ for some involution $i$ of $\hat G'$. Let $g$ be a preimage of $i$ in $S$. Note that $g$ cannot have infinite order, otherwise it would be mapped to an element of order $1$ or $p$ in $\hat F$. Thus $g$ is an involution and the projection $G'  \twoheadrightarrow \hat G'$ induces an isomorphism from $F=\langle g\rangle$ onto $\hat F=\langle i\rangle$.\end{proof}

\begin{lemma}
\label{norm stab v is in stab v}
Let $\hat g \in \text{Stab}(\hat v)$ for some $\hat v \in \hat v(\mathcal{Q})$ be of order $>2$. Then, $\hat v$ is the only apex of $\hat v(\mathcal{Q})$ fixed by $\hat g$.
\end{lemma}

\begin{proof}Recall that there is a primitive loxodromic translation $h$ such that $\text{Stab}(\hat v)$ is the quotient of $\text{Stab}(Y_{h})$ by its subgroup normally generated by $h^p$, and $\text{Stab}(Y_{h})$ is dihedral infinite by Lemma \ref{loxodromic subgroups in stprimezero} (with $\langle h\rangle$ as a subgroup of index 2). In particular, the only elliptic elements of $\text{Stab}(Y_{h})$ are the involutions, and thus $\hat g$ is not the image of an elliptic element of $\text{Stab}(Y_{h})$. Therefore, by Lemma~\ref{image out stab y fixes uniq vert}, $\hat v$ is the unique apex of $\hat v (\mathcal{Q})$ fixed by $\hat g$.
\end{proof}

We can now prove the following result (which is the first part of Proposition \ref{reformulation class cprime stable under sc quotients}).

\begin{proposition}
\label{prop: classcprime}
If $(G', X')\in\classcprimezero$, then so is $(\hat{G}', \hat{X}')$.
\end{proposition}




\begin{proof}We verify that $(\hat{G}', \hat{X}')$ satisfies the conditions of Definition~\ref{definition class Cprimezero} of $\classcprimezero$.\medskip

\noindent{\bf \ref{def class Cprime no parabolic}}: the action of $\hat{G}'$ on $\hat{X}'$ is acylindrical.

\medskip

\label{action of ghat acyl}This is a consequence of Theorem~\ref{acyl passes to sc quotient} and Lemma~\ref{apex stab in SC quot}, since this lemma implies that for all $h \in P$, we have $\lvert \text{\text{Stab}}(Y_{h})/\langle h^{p} \rangle \rvert=2p$. 

\medskip

\label{sc quotient trans of inf order lox}\noindent{\bf \ref{def class Cprime inf trans are loxodromic}}: translations of infinite order of $\hat G '$ are loxodromic for their action on $\hat X '$.

\medskip

Assume towards a contradiction that there exist two distinct involutions $i\neq j$ in $\hat G '$ such that $ij$ is elliptic and has infinite order. Since $\langle ij \rangle$ is a subgroup of finite index of $D_{ i, j }$ this last subgroup is also elliptic, and by Lemma~\ref{elliptic maps to Dp} it must lift to an elliptic subgroup of $G'$. This would yield an elliptic translation of $G'$ of infinite order, contradicting Condition~\ref{def class Cprime inf trans are loxodromic} for the pair $(G',X')$ and thus the assumption that it belongs to $\classcprimezero$. 

\medskip
    
\noindent{\bf \ref{def class Cprime no subg of odd ord norm by lox}}: no loxodromic element of $\hat G '$ normalizes a non-trivial finite subgroup $\hat F$ of order $>2$.

\medskip

\label{no lox elem norm a subg of odd order}This follows directly from Lemma~\ref{loxodromic subgroups in the quotient}. Indeed, since the pair $(G',X')$ is in class $\classcprimezero$, every maximal normal finite subgroup of a loxodromic subgroup of $G'$ is of order at most two. Moreover, the maximal normal finite subgroup of every $\text{Stab}(Y)$ for $(H,Y) \in \mathcal{Q}$ is trivial, since it contains two involutions and thus by Lemma~\ref{loxodromic subgroups in stprimezero} it has to be isomorphic to $D_{\infty}$.

\begin{remark}
\label{loxodromic subg G' sc quotient}
    Since we have proved that the pair $(\hat G ' , \hat X ')$ satisfies Condition~\ref{def class Cprime no subg of odd ord norm by lox} of Definition~\ref{definition class Cprime} and Condition~\ref{def class C pairs of minimal or affine type} of Definition~\ref{definition anss2t}, the loxodromic subgroups of $\hat G '$ also satisfy the classification exhibited in Lemma~\ref{loxodromic subgroups in stprimezero}.
\end{remark}
\noindent{\bf \ref{def class Cprimezero satisfies the ind lemma}}:\label{tuple of cprime satisfies ind hyp}
the tuple $(\hat{G}', \hat{N}', \hat{Q}', \hat{X}')$ satisfies the assumptions of Proposition~\ref{res: SC - induction lemma} for $\nu_0 =5$ and exponent $p$ (except that $\hat{Q}'$ is not required \textit{a priori} to be stable with respect to a family $\hat{\mathcal{Q}}'$).

\medskip

It only remains to be proved that $Q'$ is stable with respect to $\mathcal{Q}$ to get the bound on the injectivity radius of $\hat{Q}'$.
Let $\hat{g}$ be a loxodromic element of the image $\hat{Q}'$ of $Q'$ in $\hat G '$. As a non-trivial image of a translation, it is itself a loxodromic translation. As such, for every $m \in \mathbb{N}$ we will have that $\text{Stab}(A_{\hat{g}^{m}})$ is the maximal loxodromic subgroup containing $\hat{g}$. By Remark~\ref{loxodromic subg G' sc quotient}, this has to be an infinite dihedral group. Thus, if we have a subset $A$ of $\dot{X}'$ such that the projection $G' \twoheadrightarrow \hat{G}'$ induces an isomorphism from $\text{Stab}(A)$ onto $\text{Stab}(A_{\hat{g}^{m}})$, then $\text{Stab}(A)$ is an infinite dihedral group and thus the preimage $g$ of $\hat{g}$ in $\text{Stab}(A)$ is a translation of infinite order. Thus, $g \in Q'$ and indeed the family $Q'$ is stable with respect to $\mathcal{Q}$.

\medskip

We now show that $(\hat{G}',\hat{X}')$ is almost sharply 2-transitive of characteristic $p$. 

\medskip

\noindent{\bf \ref{def class C cent of trans is cyclic}}: for $(\hat r, \hat s) \in I(\hat G')^{(2)}$, $\text{Cen}(\hat r \hat s)$ is cyclic and generated by a translation.

\medskip

Since $(\hat G', \hat X')$ satisfies Condition~\ref{def class Cprime inf trans are loxodromic} a translation $\hat r\hat s$ either has finite order or is loxodromic.

Let $\hat h \in \text{Cen}(\hat r \hat s) $. If $\hat r \hat s$ is of finite order, then $\hat h$ cannot be loxodromic since $(\hat{G}', \hat{X}')$ satisfies Condition~\ref{def class Cprime no subg of odd ord norm by lox}, namely no loxodromic element of $\hat G'$ normalizes a non-trivial finite subgroup of order $>2$. Also, if $\hat r \hat s$ is of finite order and $\hat h$ is elliptic, then so is $\hat E=\langle \hat r \hat s, \hat h \rangle$ since it contains $\langle \hat h \rangle$ as a finite index subgroup. We now distinguish four cases.

\textit{Case 1}: $D_{\hat r, \hat s}$ is loxodromic. Then $\hat h$ is in the maximal loxodromic subgroup containing $D_{\hat r, \hat s}$. By Remark~\ref{loxodromic subg G' sc quotient}, this subgroup is isomorphic to $D_{\infty}$, where the centralizer of any translation is cyclic and generated by a translation.

\textit{Case 2}: $D_{\hat r, \hat s}$ is elliptic and does not lift. Then there is some $\hat v \in \hat v(\mathcal{Q})$ with $D_{\hat r, \hat s}=\text{Stab}(\hat v)$, and $\hat v$ is the unique apex fixed by $\hat r \hat s$ by Lemma~\ref{norm stab v is in stab v}. Since $\hat h$ normalizes the subgroup $\langle \hat r \hat s \rangle$ we get $\hat h \in \text{Stab}(\hat v)$, and in this subgroup the centralizer of any translation is cyclic and generated by a translation.

\textit{Case 3}: $D_{\hat r, \hat s}$ is elliptic and lifts, but $\hat E$ does not lift. Note that $\hat E$ is elliptic (indeed, $\hat r\hat s$ has finite order, $\hat{h}$ is necessarily elliptic and $\langle \hat h \rangle$ has finite index in $\hat E$). Then there is some $\hat v \in \hat v(\mathcal{Q})$ with $\hat E \subset \text{Stab}(\hat v)$. But $\text{Stab}(\hat v)$ is isomorphic to $D_{p}$, where the centralizer of any translation is cyclic and generated by a translation. 

\textit{Case 4}: $D_{\hat r, \hat s}$ is elliptic and lifts, and $\hat E$ lifts. Let $r,s$ denote two involutions of $G'$ such that $D_{r,s}$ is a lift of $D_{\hat r, \hat s}$, and let $g$ be the preimage of $\hat r \hat s$ in a preimage $E$ of $\hat E$. Notice that, since $g$ maps to an element of $\hat N$ in the quotient, then we must have $g \in N$. Now, $g$ and $rs$ are elliptic elements of $N$ such that their images are conjugate in $\hat G'$, so by the fourth consequence of Proposition~\ref{res: SC - induction lemma}, they are conjugate in $G'$. In particular, $g$ is a translation centralized by the preimage $h$ of $\hat h$. Since $G'$ satisfies Condition~\ref{def class C cent of trans is cyclic}, we get that $h$ is in the cyclic subgroup $\text{Cen}(g)$, and the claim follows. 

\medskip

\noindent{\bf \ref{def class C pairs of minimal or affine type}}: every translation is either of order $p$ or of infinite order, and every pair $(\hat r, \hat s) \in I(\hat G ')^{(2)}$ such that $\hat r\hat s$ has order $p$ is either of $p$-affine type or of $p$-minimal type.

\medskip

First, let us prove that the the order of a translation of $\hat G '$ is either $p$ or infinite. Assume towards a contradiction that $i,j \in \hat G '$ are involutions such that $ij$ has finite order $\neq p$. Then by Lemma~\ref{elliptic maps to Dp} the group generated by $i$ and $j$ must lift to $G'$, contradicting Condition~\ref{def class C pairs of minimal or affine type} for $G'$.

Then, let $(\hat r, \hat s) \in I(\hat G')^{(2)}$ be such that $\hat r\hat s$ has order $p$, and let us prove that $(\hat r, \hat s)$ is either of $p$-affine type or of $p$-minimal type. Note that the subgroup $D_{\hat r, \hat s}$ is elliptic. We distinguish three cases.
    
\textit{Case 1}: the subgroup $ D_{\hat r, \hat s}$ does not lift. Then by Lemma~\ref{elliptic maps to Dp} there is some $\hat v \in \hat v (\mathcal{Q})$ such that $ D_{\hat r, \hat s}=\text{Stab}(\hat v)$. By Lemma~\ref{norm stab v is in stab v}, $\hat v$ is the only apex fixed by $\hat r\hat s$, and thus $N_{\hat G'}(\langle \hat r\hat s \rangle)$ is contained in $\text{Stab}(\hat v)$. But note that $D_{\hat r, \hat s}$ is contained in $N_{\hat G'}(\langle \hat r \hat s \rangle)$ since $\langle \hat r \hat s \rangle$ is a characteristic subgroup of $D_{\hat r, \hat s}$. Thus, $N_{\hat G'}(D_{\hat r\hat s})=D_{\hat r, \hat s}$, and therefore $(\hat r, \hat s)$ is of $p$-minimal type.

\textit{Case 2}: the subgroup $ D_{\hat r, \hat s}$ lifts and it has a preimage of $p$-affine type. Let $D_{r,s}$ be such a preimage for involutions $r,s \in G'$, then, by Remark~\ref{conj of pmin and paff is pmin paff} the normalizer $N_{G'}(D_{r,s})$ is isomorphic to $\agl$, and in particular it is elliptic in $G'$. Thus, the projection induces an isomorphism from $N_{G'}(D_{r,s})$ onto its image $H$, which contains $ D_{\hat r, \hat s}$, so $(\hat r, \hat s)$ is itself of $p$-affine type.

\textit{Case 3}: the subgroup $D_{\hat r, \hat s}$ lifts and all its preimages are of $p$-minimal type. Let us assume towards a contradiction that $D_{\hat r, \hat s}$ is not of $p$-minimal type, that is, that there is some element $\hat h \in N_{\hat G}(\langle \hat r \hat s \rangle) \backslash D_{\hat r, \hat s}$. Notice that $\hat h$ cannot be loxodromic, otherwise it would normalize the subgroup $\langle \hat r \hat s \rangle$ of order $>2$, contradicting Condition~\ref{def class Cprime no subg of odd ord norm by lox} for $(\hat G ', \hat X ')$. Therefore, $\hat h$ is elliptic, and since $\langle \hat h \rangle$ is a finite index subgroup of $\hat E=\langle \hat h, \hat r \hat s \rangle$, this subgroup is also elliptic.

If $ \hat E$ lifts, let $r,s$ denote two involutions of $G'$ such that $D_{r,s}$ is a lift of $D_{\hat r, \hat s}$, and let $g$ (respectively $h$) be the preimage of $\hat r \hat s$ (respectivley $\hat h$) in a preimage $E$ of $\hat E$. Since $g$ maps to an element of $\hat N '$ in the quotient, we must have $g \in N '$. Now, $g$ and $rs$ are elliptic elements of $N'$ such that their images are conjugate in $\hat G '$, so by the last consequence of Proposition~\ref{res: SC - induction lemma} they are conjugate in $G'$. In particular, $g$ is a translation. Let $a \in G'$ be such that $r'=a^{-1}ra$ and $s'=a^{-1}sa$ and $g=r's'$. Recall that by assumption all the preimages of $(\hat r,\hat s)$ are of $p$-minimal type, so $(r,s)$ and thus $(r',s')$ are of $p$-minimal type. It follows that $h$ is in $D_{r',s'}$. Moreover, since both $rs$ and $r's'$ are preimages of $\hat r \hat s$ in $G'$, then $\hat a$ is an element of $\hat G'$ centralizing the translation $\hat r \hat s$, and by Condition~\ref{def class C cent of trans is cyclic} proved above,  $\hat a$ belongs to $\langle \hat r \hat s \rangle$. In particular, the images of $r'$ and $s'$ in $\hat G'$ are contained in $D_{\hat r, \hat s}$, and therefore so is $\hat h$, contradicting our initial assumption that $\hat h$ does not belong to $D_{\hat r, \hat s}$.

If $\hat E$ does not lift, then by item (\ref{elliptic subgroups of hat g}) of Lemma~\ref{small cancellation theorem} there is some $\hat v \in \hat v(\mathcal{Q})$ such that $\hat E$ is contained in $\text{Stab}(\hat v)$, and by Lemma~\ref{norm stab v is in stab v} the apex $\hat v$ is the only apex fixed by $\langle \hat r \hat s \rangle$. Hence the normalizer of $\langle \hat r \hat s \rangle$ fixes $\hat v$, in particular $D_{\hat r, \hat s}$ fixes $\hat v$. By Lemma~\ref{elliptic maps to Dp} we have $\hat E\subset\text{Stab}(\hat v)=D_{\hat r, \hat s}$. Therefore $\hat h$ belongs to $D_{\hat r, \hat s}$, contradicting our assumption. 

\medskip

\noindent{\bf \ref{def class C N without 2-torsion}}: $\hat{N} '$ contains no involution.

\medskip

This is contained in Proposition~\ref{res: SC - induction lemma}. 

\medskip

\noindent{\bf \ref{def class C G trans on affine trans}}: the set of pairs $(r,s) \in I(\hat G')^{(2)}$ of $p$-affine type is non-empty and $\hat G '$ acts transitively on it.

\medskip

Since $G'$ contains a subgroup isomorphic to $\agl$, and since the projection from $G'$ onto $\hat G'$ induces an isomorphism from any elliptic subgroup onto its image, the group $\hat G'$ contains a subgroup isomorphic to $\agl$ as well, and thus the set of pairs of involutions of $p$-affine type of $\hat G'$ is non-empty.

Let now $(\hat r, \hat s)$ and $(\hat r', \hat s')$ be two pairs of $I(\hat G')^{(2)}$ of $p$-affine type. By Lemma~\ref{elliptic maps to Dp}, any subgroup of $\hat G'$ isomorphic to $\agl$ lifts, and thus there are pairs $(r,  s)$ and $(r', s')$ of $I( G')^{(2)}$ of $p$-affine type lifting $(\hat r, \hat s)$ and $(\hat r', \hat s')$ respectively. Since Condition~\ref{def class C G trans on affine trans} holds for $G'$, there exists $g \in G$ such that $g^{-1}rg=r'$ and $g^{-1}sg=s'$. Denoting by $\hat g$ the image of $g$ in $\hat G'$, we get $\hat g^{-1} \hat r \hat g= \hat r'$ and $\hat g^{-1} \hat s \hat g= \hat s'$.

Hence $\hat{G}'$ is almost sharply 2-transitive and we have established Proposition~\ref{prop: classcprime}.\end{proof}

The proof of Condition \ref{def class Cprimezero satisfies the ind lemma} in Proposition \ref{prop: classcprime} shows that, as it was stated before, the stability of the family $Q'$ with respect to $\mathcal{Q}'$ follows from the other conditions of Definition \ref{definition class Cprimezero}. In fact, it also shows the following result.

\begin{corollary}\label{cor strongly stable}
If $(G', X')$ belongs to $\classcprimezero$, then $Q'$ is strongly stable. 
\end{corollary}

We finish the proof of Proposition~\ref{reformulation class cprime stable under sc quotients} by showing the remaining statements in the following two lemmas.
\begin{lemma}
\label{og is ohat g}
    $O(G')=O(\hat G')$.
\end{lemma}

\begin{proof}Let $g \in G'$ be an element of finite order. In particular, $g$ is elliptic. By the second consequence of Proposition~\ref{res: SC - induction lemma}, the projection $G' \twoheadrightarrow \hat G'$ induces an isomorphism from $\langle g \rangle$ onto $\langle \hat g \rangle$, so $\lvert g \rvert = \lvert \hat g \rvert$ and $O(G') \subseteq O(\hat G')$.

    Conversely, if $\hat g \in \hat G'$ is of finite order, then it is elliptic, hence by Lemma~\ref{elliptic maps to Dp} either there is $g \in G'$ such that the projection induces an isomorphism from $\langle g \rangle$ onto $\langle \hat g \rangle$, in which case $\lvert \hat g \rvert \in O(G')$, or $\langle \hat{g}\rangle$ is isomorphic to $C_p$. However, since $G'$ is almost sharply 2-transitive, it contains a translation of order $p$, so we obtain $O(\hat G') \subseteq O(G')$.
\end{proof}

\begin{lemma}
\label{sc p affine and min maps to p affine and min}
    The images of pairs of distinct involutions of $p$-affine type (respectively, of $p$-minimal type) are again of $p$-affine type (respectively, of $p$-minimal type).
\end{lemma}

\begin{proof}
    Notice now that for a pair $( r, s) \in I( G')^{(2)}$ of $p$-affine type, the projection $G' \twoheadrightarrow \hat G'$ induces an isomorphism from $N_{G'}(D_{r,s})$ onto its image (since this subgroup is finite by Remark~\ref{conj of pmin and paff is pmin paff}), and thus its image $(\hat r, \hat s)$ will be itself of $p$-affine type.

    Moreover, suppose that a pair $(r,s)$ of $p$-minimal type maps in the quotient to a pair $(\hat r, \hat s)$ of $p$-affine type. By Lemma~\ref{elliptic maps to Dp}, any subgroup of $\hat G'$ isomorphic to $\agl$ lifts, thus there is a pair $(r',s') \in I( G')^{(2)}$ of $p$-affine type lifting $(\hat r, \hat s)$. However, in this case, $rs$ and $r's'$ are elliptic elements of $N'$ such that their images are conjugate in $\hat G'$, so they must themselves be conjugate, say by an element $g \in G'$. But then, we would have $ \lvert D_{r,s} \cap g^{-1}D_{r',s'}g \rvert \geq p$, contradicting Lemma~\ref{qm pair}.\end{proof}


This finishes the proof of Proposition~\ref{reformulation class cprime stable under sc quotients}.

\subsection[PP-Quotients]{Stability of the Classes under PP-Quotients}
\label{subsec pp quotient}





We are ready to prove Proposition~\ref{prop: classes stable under pp quotients}. Recall that throughout this section, we fix $(G,X)\in\classcprime$. By Corollary~\ref{cor strongly stable} we can apply Proposition~\ref{res : SC - partial periodic quotient} and obtain a quotient $\bar{G}$ as the limit of iterated applications of Proposition~\ref{prop: classcprime}.
It remains to prove that $\bar{G}$ is almost sharply 2-transitive of characteristic $p$.

We now state an easy remark that will allow us to lift relations in the partial periodic quotient $\bar G$ to relations in some intermediate small cancellation quotient $(\hat G_{m}, \hat X_{m})$.

\begin{remark}
\label{from sc to pp}
    Let $\bar g^{(1)},...,\bar g^{(a)} \in \bar G$ be such that $\bar g^{(1)}...\bar g^{(a)}=1$. Then, there exist $m \in \mathbb{N}$ and preimages $\hat g^{(i)}_{m}$ of the elements $\bar g^{(i)}_{m}$ (for $1\leq i\leq a$) in the small cancellation quotient $\hat G_{m}$ of $G$ such that $\hat g^{(1)}_{m}...\hat g^{(a)}_{m}=1$.
\end{remark}

Moreover, for all $m \in \mathbb{N}$, the small cancellation quotient pair $(\hat G_{m}, \hat X_{m})$ is obtained by applying $m$ times Proposition~\ref{reformulation class cprime stable under sc quotients}, so this pair is in class $\classcprimezero$, and the images of pairs of distinct involutions of $p$-affine type (respectively, of $p$-minimal type) will be themselves of $p$-affine type (respectively, of $p$-minimal type). The following lemma shows that this is still true after taking the limit of the sequence. 

\begin{lemma}
\label{pp pmin and paff maps to pmin and paff}
    The images in the partial periodic quotient $\bar G$ of pairs of distinct involutions of $p$-affine type (respectively, of $p$-minimal type) will be themselves of $p$-affine type (respectively, of $p$-minimal type). Moreover, the images of pairs of involutions that generate a loxodromic subgroup will be of $p$-minimal type.
\end{lemma}

\begin{proof}
    Notice first that, by construction, all translations of $\bar G$ will be of order $p$ (by consequence~(\ref{pp quotient image of N is partial periodic}) of Theorem~\ref{res : SC - partial periodic quotient}).

    Let $(r,s) \in I(G)^{(2)}$ be a pair of $p$-affine type. Then, by the first consequence of Theorem~\ref{res : SC - partial periodic quotient}, the projection $G \twoheadrightarrow \bar G$ induces an isomorphism from $N_{G}(D_{r,s})$ onto its image (since by Remark~\ref{conj of pmin and paff is pmin paff} this subgroup is isomorphic to $\agl$), so the pair $(\bar r, \bar s)$ will be of $p$-affine type.

    Let $(r,s) \in I(G)^{(2)}$ be a pair of $p$-minimal type. Suppose that there is some element $\bar g \in N_{\bar G}(\langle \bar r \bar s \rangle) \backslash D_{\bar r, \bar s}$. In particular, $\bar g$ conjugates $\bar r \bar s$ to $(\bar r \bar s)^{l}$ for some $l \in \{ 1, \dots , p-1 \}$. Therefore, there is some $m \in \mathbb{N}$ such that a preimage $\hat g_{m}$ of $\bar g$ in $\hat G_{m}$ conjugates a preimage $\hat r_{m} \hat s_{m}$ of $\bar r \bar s$ to $(\hat r_{m} \hat s_{m})^{l}$. Clearly $\hat g_{m}$ cannot be in $D_{\hat r_{m}, \hat s_{m}}$ (since that would yield $\bar g \in D_{\bar r, \bar s}$). That is, $(\hat r_{m}, \hat s_{m})$ is not of $p$-minimal type, and this is a contradiction to Proposition~\ref{reformulation class cprime stable under sc quotients}, since this implies that the images in $\hat G_{m}$ of pairs of $p$-minimal type of $G$ are themselves of $p$-minimal type.

    Let $(r,s) \in I(G)^{(2)}$ be a pair of involutions generating a loxodromic subgroup. Then, there is some $m \in \mathbb{N}$ such that $ D_{\hat r_{m}, \hat s_{m}} \cong D_{\infty}$ and $ D_{\hat r_{m+1}, \hat s_{m+1}} \cong D_{p}$. Thus, there must be some $\hat v \in \hat v (\mathcal{Q}_{m})$ such that $D_{\hat r_{m+1}, \hat s_{m+1}} \subset \text{Stab}(\hat v)$. Now, $\text{Stab}(\hat v) \cong D_{p}$, so $D_{\hat r_{m+1}, \hat s_{m+1}} = \text{Stab}(\hat v)$. Lemma~\ref{norm stab v is in stab v} readily implies that $D_{\hat r_{m+1}, \hat s_{m+1}}$ is of $p$-minimal type, and then an argument completely analogous to the one in the previous paragraph yields that the pair $(\bar r, \bar s)$ is of $p$-minimal type itself.
\end{proof}

We also note the following easy fact.

\begin{lemma}
\label{og to obarg}
    $O(G)=O(\bar G)$.
\end{lemma}

\begin{proof}
    Let $g \in G$ be an element of finite order. In particular, $g$ is elliptic. By the first consequence of Theorem~\ref{res : SC - partial periodic quotient}, the projection $G \twoheadrightarrow \bar G$ induces an isomorphism from $\langle g \rangle$ onto $\langle \bar g \rangle$, so $\lvert g \rvert =\lvert \bar g \rvert$ and $O(G) \subseteq O(\bar G)$.

    Conversely, let $\bar g \in \bar G$ be an element of finite order. By Remark~\ref{from sc to pp}, there is some $m \in \mathbb{N}$ and a preimage $\hat g_{m}$ of $\bar g$ in $\hat G_{m}$ such that $\lvert \bar g \rvert = \lvert \hat g_{m} \rvert$. Now, by applying Lemma~\ref{og is ohat g} $m$ times, we get $\lvert \bar g \rvert \in O(G)$.
\end{proof}

We can now establish the following result.

\begin{proposition}\label{prop: G bar is ash2trans}
 The partial periodic quotient $\bar{G}$ is an almost sharply 2-transitive group of characteristic $p$.   
\end{proposition}
\begin{proof}
We verify the conditions from Definition~\ref{definition anss2t}.

\medskip

\noindent\label{pp translations of order p}{\bf \ref{def class C pairs of minimal or affine type}}: every translation is either of order $p$ or of infinite order, and every pair $(\bar r,\bar s) \in I(\bar G)^{(2)}$ such that $\bar r\bar s$ is of order $p$ is either of $p$-minimal type or of $p$-affine type.

\medskip

This is an immediate consequence of Lemma~\ref{pp pmin and paff maps to pmin and paff} since every pair of $I(G)^{(2)}$ is either of $p$-minimal type, of $p$-affine type, or generates a loxodromic subgroup.

\medskip

\noindent\label{pp trans on aff type}{\bf \ref{def class C G trans on affine trans}}: the set of pairs $(\bar r, \bar s) \in I(\bar G)^{(2)}$ of $p$-affine type is non-empty and $\bar G$ acts transitively on it.

\medskip

Since the image of a pair of involutions of $p$-affine type is of $p$-affine type, there will be one such pair in $\bar G$.
    
Fix now two pairs $(\bar r, \bar s), (\bar r', \bar s') \in I(\bar G)^{(2)}$ of $p$-affine type. There is some $m \in \mathbb{N}$ and preimages $\hat r_{m}, \hat s_{m}, \hat r'_{m}, \hat s'_{m}$ of $\bar r, \bar s, \bar r', \bar s'$ respectively such that $\hat r_{m} \hat s_{m}$ and $ \hat r'_{m} \hat s'_{m}$ have order $p$. Now, by Lemma~\ref{pp pmin and paff maps to pmin and paff}, the pairs $(\hat r_{m}, \hat s_{m}), (\hat r'_{m}, \hat s'_{m})$ must be of $p$-affine type (otherwise $(\bar r, \bar s) $ and $ (\bar r', \bar s')$ would be of $p$-minimal type), so there must be an element $\hat g_{m}$ conjugating one pair to the other (since $\hat{G}_{m}$ is in class $\classcprimezero$), and the image $\bar g$ of $\hat g_{m}$ in $\bar G$ conjugates $(\bar r, \bar s)$ to $ (\bar r', \bar s')$. 

\medskip

\noindent{\bf \ref{def class C cent of trans is cyclic}}: let $\bar r \bar s$ be a translation of the partial periodic quotient $\bar G$. Then, $\text{Cen}(\bar r \bar s)= \langle \bar r \bar s \rangle$.

\medskip

Assume towards a contradiction that, for a pair $(\bar r, \bar s) \in I(\bar G)^{(2)}$, we have an element $\bar g \in \text{Cen}(\bar r \bar s) \backslash \langle \bar r \bar s \rangle$. By Remark~\ref{from sc to pp}, there are $m,m' \in \mathbb{N}$ such that some preimages $\hat r_{m}$ and $\hat s_{m'}$ of $\bar r$ and $\bar s$ are involutions in $\hat G_{m}$ and $\hat G_{m'}$ respectively. By the second consequence of Proposition~\ref{res: SC - induction lemma}, this relations will be preserved for all $n \geq m$, $n' \geq m'$. Moreover, there exists $l \in \mathbb{N}$ such that a preimage $\hat r_{l} \hat s_{l}$ of $\bar r \bar s$ commutes with some preimage $\hat g_{l}$ of $\bar g$, and once again this relation will be preserved for all $n'' \geq l$. Take now $m''=\max (\{ m,m',l \})$. Then, in $\hat G_{m''}$ we have an element $\hat g_{m''}$ centralizing the translation $\hat r_{m''} \hat s_{m''}$. Moreover, $\hat g_{m''}$ cannot be in $\langle \hat r_{m''} \hat s_{m''} \rangle$, since otherwise this relation would still hold for their images $\bar g$ and $\bar r \bar s$ in $\bar G$. This contradicts the fact that the pair $(\hat G_{m''}, \hat X_{m''})$ belongs to class $\classcprimezero$. 

\medskip

\noindent{\bf \ref{def class C N without 2-torsion}}:  the subgroup $\langle Tr(\bar G) \rangle$ contains no involution.

\medskip

Suppose towards a contradiction that there are involutions $\bar r^{(1)}, \bar s^{(1)},..., \bar r^{(i)}, \bar s^{(i)}$ for $i \in \mathbb{N}$ such that $\bar r^{(1)} \bar s^{(1)}... \bar r^{(i)} \bar s^{(i)}$ is an involution. By Remark~\ref{from sc to pp}, there are $m_{j},m'_{j} \in \mathbb{N}$, $1 \leq j \leq i$ such that some preimages $\hat r^{(j)}_{m_{j}}$, $\hat s^{(j)}_{m'_{j}}$ in $\hat G_{m_{j}}$ and $\hat G_{m'_{j}}$ of $\bar r^{(j)}$ and $\bar s^{(j)}$ respectively are involutions, and this will remain the case for all $n_{j} \geq m_{j}$, $n'_{j} \geq m'_{j}$. Moreover, there is some $m'' \in \mathbb{N}$ such that some preimage $\hat r^{(1)}_{m''} \hat s^{(1)}_{m''}... \hat r^{(i)}_{m''} \hat s^{(i)}_{m''}$ in $\hat G_{m''}$ of $\bar r^{(1)} \bar s^{(1)}... \bar r^{(i)} \bar s^{(i)}$ is an involution, and this will remain the case for all $n''\geq m''$. Put $m'''=\max (\{ m_{1},m'_{1},...,m_{i},m'_{i},m'' \})$, then we get an element in $\langle Tr(\hat G_{m'''}) \rangle$ that is an involution, and this is a contradiction to the fact that the pair $(\hat G_{m'''}, \hat X_{m'''})$ is in class $\classcprimezero$.   
\end{proof}
Hence, the group $\bar{G}$ is almost sharply $2$-transitive of characteristic $p$. Moreover, by construction, $\bar{G}$ has no translation of infinite order, and thus it belongs to the class $\classc$. This finishes the proof of Proposition~\ref{prop: classes stable under pp quotients}.

We end this section with the following result, which implies (by Theorem \ref{if splits trans abelian}) that the groups constructed in Theorem \ref{embedding theorem} do not split.

\begin{lemma}
\label{non-commuting translations}
    $\bar{G}$ contains non-commuting translations.
\end{lemma}

\begin{proof}
Consequence (\ref{pp quotient infinitely many elem in bar N}) of Theorem \ref{res : SC - partial periodic quotient} implies in particular that $\bar{N}$ is infinite. Since this subgroup is generated by translations of $\bar{G}$, and in this group the centralizer of every translation is cyclic of order $p$ (by Condition \ref{def class C cent of trans is cyclic} of the definition of class $\classc$) then $\bar{N}$ must have a pair of translations that do not commute.\end{proof}

\section[Second Quotient]{Proof of Proposition~\ref{prop:ext in cprime}}
\label{section second quotient}
In this section we prove Proposition~\ref{prop:ext in cprime}, which we recall for convenience.

\begin{proposition2.14} Let $p$ be an odd prime number and let $G$ be a group in $\classc$. 
\begin{enumerate}
\item  Let $(r,s), (r',s') \in I(G)^{(2)}$ with $(r,s)$ of  $p$-affine type, and $(r',s')$ of $p$-minimal type and \[\gstar = \langle G,t \ \vert \ trt^{-1}=r', tst^{-1}=s'\rangle.\]
\item Let $H$ be a non-trivial group without $2$-torsion and  $\gstar= G * H$.
\end{enumerate}   
Let $X$ denote the corresponding Bass-Serre tree. Then $(\gstar,X)$ belongs to $\classcprime$.
\end{proposition2.14}

We first observe the following general result.
\begin{lemma}
\label{pathstabsmall}
Let $G$ be a group and let $(K,K')$ be a quasi-malnormal pair (see Definition~\ref{def quasi-malnormal}) of isomorphic subgroups of $G$. Let $\alpha : K\rightarrow K'$ be an isomorphism, and consider the group $\gstar=\langle G,t \ \vert \ tkt^{-1}=\alpha(k)\rangle$. Let $X$ be the corresponding Bass-Serre tree. Then stabilizers of paths with at least three edges in $X$ have order at most two.
\end{lemma}

\begin{proof}Let $v$ denote the vertex of $X$ fixed by $G$. Let $[v_1,v_2,v_3,v_4]$ be a path of length 3 in the tree $X$ (with $d(v_i,v_{i+1})=1$). We will prove that the (global) stabilizer of this path has order at most 2. Since $\gstar$ acts transitively on the (non-oriented) edges of $X$, we can assume that $[v_2,v_3]=[v,t^{-1}v]$ or $[v_2,v_3]=[t^{-1}v,v]$. Without loss of generality, suppose that $[v_2,v_3]=[v,t^{-1}v]$. The stabilizer of this edge is $G\cap t^{-1}Gt=K$. Since the vertices $v_1$ and $v_2=v$ are adjacent, there exists an element $g\in G$ such that $v_1=gt^{-1}v$ or $v_1=gtv$, thus the stabilizer of the path $[v_1,v_2,v_3]$ is $gKg^{-1}\cap K$ or $gK'g^{-1}\cap K$, hence it has order at most 2 since the pair $(K,K')$ is quasi-malnormal.\end{proof}

\begin{proof}[Proof of Proposition~\ref{prop:ext in cprime}]
We treat both cases uniformly, so let $G\in\classc$, and let $\gstar$ denote either the HNN extension or the free product defined in Proposition~\ref{prop:ext in cprime}, let $X$ denote the corresponding Bass-Serre tree and define $Q=Tr(\gstar)$ and $N=\langle Tr(G) \rangle$. We need to show that $G$ is almost sharply 2-transitive of characteristic $p$ and satisfies the conditions given in Definition~\ref{definition class Cprime} for $\classcprime$.

\medskip

To see that $G^*$ is almost sharply 2-transitive of characteristic $p$, we first verify
Condition~\ref{def class Cprime no subg of odd ord norm by lox} of Definition~\ref{definition class Cprime} of $\classcprime$ as this condition will be used in the proof of almost sharply 2-transitivity.

\medskip

\noindent\label{order_at_most_2}{\bf \ref{def class Cprime no subg of odd ord norm by lox}}: no loxodromic element of $\gstar$ normalizes a non-trivial finite subgroup of order $>2$.

\medskip

 Let $F$ be a finite subgroup of $\gstar$ of order $>2$. We will prove that $N_{\gstar}(F)$ is elliptic in the Bass-Serre tree $X$. First, note that the fixed-point set $Y\subset X$ of $F$ is preserved by $N_{\gstar}(F)$.
 
 If $\gstar=G\ast H$, then $Y$ is a point, and thus $N_{\gstar}(F)$ fixes this point as well. 
 
 If $\gstar$ is the HNN extension defined in Proposition~\ref{prop:ext in cprime} then, by Lemmas~\ref{qm pair} and~\ref{pathstabsmall}, the subtree $Y$ has diameter at most two. If $Y$ is a point, then $N_{\gstar}(F)$ fixes this point. If $Y$ has diameter one, then it consists of two vertices joined by an edge $e$, and $N_{\gstar}(F)$ fixes the edge $e$ as well. If $Y$ has diameter two, then $N_{\gstar}(F)$ fixes the midpoint of this subtree.

\medskip

We now show that $G^*$ is almost sharply 2-transitive as defined in Definition~\ref{definition anss2t}.

\medskip

\noindent\label{all trans of inf or p order in gstar}{\bf \ref{def class C pairs of minimal or affine type}}: every translation is either of order $p$ or of infinite order, and every pair $(r,s) \in I(\gstar)^{(2)}$ such that $rs$ has order $p$ is either of $p$-affine type or of $p$-minimal type.

\medskip

Let $g=rs$ be a translation, where $r,s$ denote two distinct involutions of $\gstar$. If $r$ and $s$ have a common fixed in the tree $X$, then their product $g$ is contained in a conjugate of $G$, and thus it has order $p$ and it is either of $p$-affine type or of $p$-minimal type since $G$ belongs to $\classc$ (and by Condition~\ref{def class Cprime no subg of odd ord norm by lox} above, the normalizer of the subgroup it generates is elliptic). If $r$ and $s$ do not have a common fixed point, then $g$ is loxodromic (and thus it has infinite order).

\medskip
\noindent
{\bf \ref{def class C N without 2-torsion}}: the subgroup $N^*=\langle Tr(\gstar) \rangle$ contains no involution.

\medskip

First, suppose that $\gstar=G\ast H$, where $H$ has no involution, as in Proposition~\ref{prop:ext in cprime}. Assume towards a contradiction that there is an involution in $N^*$, or equivalently, an odd number of involutions $r_{1},\ldots,r_{k}$ of $\gstar$ such that $r_{1}\cdots r_{k}=1$. Since $H$ has no involution, every $r_i$ is conjugate to an involution of $G$. Let $\pi^{\star} : \gstar \twoheadrightarrow G$ denote the map obtained by killing the free factor $H$. Then $\pi^{\star}(r_{1}\cdots r_{k})=1$, and $\pi^{\star}(r_{1}\cdots r_{k})$ is a product of an odd number of involutions in $G$, contradicting the assumption that $N$ has no involution (since $G$ belongs to $\classc$). 

Then, suppose that $\gstar$ is the HNN extension defined in Proposition~\ref{prop:ext in cprime}. The argument is the same as in the previous case, but a bit more involved. Assume towards a contradiction that there is an involution in $N^*$, or equivalently, an odd number of involutions $r_{1},\ldots,r_{k}$ of $\gstar$ such that $r_{1}\cdots r_{k}=1$. Consider the quotient of $\gstar$ by the normal subgroup generated by the stable letter $t$ of the HNN extension, that is, the group $G'= \gstar / \langle t \rangle ^{\gstar}$, and call $\pi^{\star} : \gstar \twoheadrightarrow G'$ the quotient map. Notice that $G'$ can also be obtained as the quotient of $G$ by the relations $k=\alpha(k)$, that is $G' \cong G / \langle k^{-1}\alpha(k) \, : \, k \in K \rangle ^{G}$. We will identify these isomorphic groups, and call $ \pi: G \twoheadrightarrow G'$ the quotient map. Note that the kernel of $\pi$ is contained in $N=\langle Tr(G) \rangle$. By assumption $G$ belongs to $\classc$, thus $N$ contains no involution. Hence no involution of $G$ is mapped to the identity by $\pi$. Furthermore, the restriction of $\pi^{\star}$ to $G$ coincides with $\pi$, and every involution of $\gstar$ is conjugate to an involution of $G$. Therefore, no involution of $\gstar$ is mapped to the identity by $\pi^{\star}$. Thus, the image $\pi ^{\star}(r_{1}...r_{k})$ is a (trivial) product of an odd number of involutions of $G'$. Moreover, every preimage of $\pi^{\star}(r_{1}...r_{k})$ in $G$ will be a product of an odd number of involutions. That is, there is an involution in $N=\langle Tr(G) \rangle$, and we arrive at a contradiction.

\medskip

\noindent{\bf \ref{def class C G trans on affine trans}}: the set of pairs $(r,s) \in I(\gstar)^{(2)}$ of $p$-affine type is non-empty and $\gstar$ acts transitively on it.

\medskip

First, note that this set is non-empty since $G$ is contained in $\gstar$. Then, if $(r,s) \in I(\gstar)^{(2)}$ is of $p$-affine type then $\langle r,s\rangle$ is finite, thus it is elliptic in the tree $X$. It follows that $\langle r,s\rangle$ is contained in a conjugate of $G$, and the conclusion follows from the fact that $G$ belongs to $\classc$.

\medskip

\noindent{\bf \ref{def class C cent of trans is cyclic}}: for $(r,s) \in I(\gstar)^{(2)}$, $\text{Cen}(rs)$ is cyclic and generated by a translation.

\medskip

If $rs$ is elliptic then its centralizer is elliptic in the tree $X$ by  Condition~\ref{def class Cprime no subg of odd ord norm by lox} above, and the conclusion follows from the fact that $G$ belongs to $\classc$.

If $rs$ is loxodromic then it is contained in a (unique) maximal loxodromic subgroup of $\gstar$, which is isomorphic to $D_{\infty}$ (indeed, exactly as in Remark \ref{loxodromic subg G' sc quotient}, since $\gstar$ satisfies Condition~\ref{def class C pairs of minimal or affine type} and Condition~\ref{def class Cprime no subg of odd ord norm by lox}, every loxodromic subgroup of $\gstar$ is isomorphic to $\mathbb{Z}$ or $C_{2}\times \mathbb{Z}$ or $D_{\infty}$ or $C_4\ast_{C_2}C_4$, and $D_{\infty}$ is the only group in this list that contains two distinct involutions), and thus the centralizer of $rs$ is cyclic infinite, generated by a translation.

\medskip

It is left to verify Conditions~\ref{def class Cprime no parabolic}--\ref{def class Cprime inf trans are loxodromic} of Definition~\ref{definition class Cprime} of $\classcprime$.

\medskip

\noindent{\bf \ref{def class Cprime no parabolic}}: the action of $\gstar$ on $X$ is acylindrical (and non-elementary).

\medskip

If $\gstar$ is the HNN extension defined in Proposition~\ref{prop:ext in cprime}, then by Lemma~\ref{qm pair} and Lemma~\ref{pathstabsmall} stabilizers of paths with at least three edges in $X$ have order at most two. If $\gstar=G\ast H$, stabilizers of edges are trivial. In both cases, one easily sees that the action of $\gstar$ on $X$ is acylindrical, and the fact that the action is non-elementary is clear.

\medskip

\noindent{\bf \ref{def class Cprime satisfies the ind lemma}}: $e(N,X)=1$, $r_{\text{inj}}(Q,X) \geq 1$, $\nu(N,X) \leq 5$ and $A(N,X)=0$.

\medskip

 Note that since the space $X$ is $0$-hyperbolic, it follows from Lemma~\ref{asympt and trans length relation} that $[g]^{\infty} = [g]$ for all $g \in \gstar$, and thus $r_{\text{inj}}(Q,X)= \inf(\{ [g] : g \in Q, \, g \ \text{loxodromic} \})$. Hence, since $X$ is a simplicial tree, we immediately get $r_{\text{inj}}(Q,X) \geq 1$ (in fact, it can be proved that $r_{\text{inj}}(Q,X)=2$).

Then, let us prove that $\nu(N,X)\leq 5$. The parameters of the definition of an acylindrical action (see Definition~\ref{acylindr definition}) corresponding to $\varepsilon = 97 \delta$ can be taken to be $L=3$ and $M=2$ in the case where $\gstar$ is the HNN extension defined in Proposition~\ref{prop:ext in cprime} (since stabilizers of paths with at least three edges in $X$ have order at most two), and $L=M=1$ in the case where $\gstar=G\ast H$. In both cases, the bound $\nu(N,X)\leq 5$ follows readily from Lemma~\ref{nu finite acyl}.

To see that  $e(N,X)=1$, by Remark~\ref{e is 1 if lox subg are cyc} it is sufficient to prove that every loxodromic subgroup $E$ of $N$ is cyclic. As in the proof of Condition \ref{def class C cent of trans is cyclic}, we have that $E$ is isomorphic to one of the following four groups: $\mathbb{Z}$, $C_{2}\times \mathbb{Z}$, $D_{\infty}$ or $C_4\ast_{C_2}C_4$. But $N^*$ has no involution (by Condition \ref{def class C N without 2-torsion}), and thus $E$ is necessarily isomorphic to $\mathbb{Z}$.

Last, let us prove that $A(N,X)=0$. First notice that, since our space $X$ is $0$-hyperbolic, the axis of an element $g$ is the set $A_{g}= \{ x \in X : d(x,g \cdot x)=[g] \}$. Also, since $\nu=\nu(N,X)$ is finite, we are looking at tuples $(g_{0}, \dots ,g_{\nu})$ of elements with $[g_{i}]=0$ for all $i \in \{ 0, \dots , \nu \}$ (that is, tuples of elliptic elements) such that they do not generate an elementary subgroup. In this case, $A_{g_{i}}$ is just the fixed-point set of $g_{i}$, and $A(g_{0}, \dots ,g_{\nu})=\text{diam}(A_{g_{0}} \cap \dots \cap A_{g_{\nu}})$ is the diameter of the intersection of their fixed-point sets. If all these elliptic elements have a fixed point in common, then every element of $\langle g_{0}, \dots , g_{\nu} \rangle$ will also fix this point, and therefore, they would generate an elementary subgroup. That is, we only need to consider tuples $(g_{0}, \dots , g_{\nu})$ of elliptic elements such that $A_{g_{0}} \cap \dots \cap A_{g_{\nu}}= \varnothing$. Thus, $A(g_{0}, \dots , g_{\nu})=0$, and therefore, $A(N,X)=0$.

\medskip

\noindent{\bf \ref{def class Cprime inf trans are loxodromic}}: translations of infinite order of $\gstar$ are loxodromic for their action on $X$.

\medskip

Let $g=rs$ be a translation, where $r,s$ denote two distinct involutions of $\gstar$. If $r$ and $s$ have a common fixed point $v\in X$, then their product $g$ fixes $v$ as well. Hence $g$ is contained in a conjugate of $G$ (namely the stabilizer of $v$ in $\gstar$), and thus it has order $p$ since $G$ belongs to $\classc$. If $r$ and $s$ do not have a common fixed point, then $g$ is loxodromic (and thus it has infinite order).

\medskip

This finishes the proof of Proposition~\ref{prop:ext in cprime} and thus of Theorem~\ref{maintheorem}.\end{proof}
\appendix

\section{Proof of Lemma~\ref{almost uniqueness of quasi-geod}}\label{appendix}

In this appendix we include the proof of Lemma~\ref{almost uniqueness of quasi-geod}. For the sake of completeness, we will also restate this result.

\begin{lemma}
\label{almost uniqueness of quasi-geod appendix}
    Let $X$ be a $\delta$-hyperbolic metric space, $x,y \in X$. Let $\gamma$ be a $(1,\ell)$-quasi-geodesic connecting $x$ to $y$ for $\ell \geq 0$. Let $p \in X$ be a point such that
    \begin{equation}
    \label{almost uniqueness of qg assumption}
        d(x,p)+d(p,y) \leq d(x,y) + 2\delta+3\ell.   
    \end{equation}
    Then, $d(p, \gamma) \leq 4\delta + 8\ell$.
\end{lemma}

\begin{proof}
    Without loss of generality, we can assume that $d(x,p) \leq d(y,p)$. If we have $d(x,p) \geq d(x,y)$, then we get that $d(y,p) \leq 2\delta+3\ell$ and we are done. Similarly, if we assume $d(y,p) \geq d(x,y)$, then we get that $d(x,p) \leq 2\delta+3\ell$. 

    Now, assume that\[\max (\{d(x,p),d(y,p) \})< d(x,y).\] Take a point $z \in \gamma$ such that
    \begin{equation}
    \label{almost uniqueness of qg point in qg}
       d(p,y) \leq d(z,y)\leq d(p,y)+\ell. 
    \end{equation}
    Since $X$ is $\delta$-hyperbolic, we have that
    \begin{equation}
    \label{almost uniqueness of qg delta hyp cond}
        \langle x,y \rangle_{p} \geq \min \{ \langle x,z \rangle_{p} \, , \, \langle y,z \rangle_{p} \} - \delta.
    \end{equation}
    Combining Definition~\ref{gromov prod def} with the second inequality of (\ref{almost uniqueness of qg point in qg}) we get
    \begin{equation}
    \label{almost uniqueness of qg dif bet gromov prod prev}
        \langle x,z \rangle_{p}-\langle y,z \rangle_{p} = \frac{1}{2}(d(x,p)-d(x,z)-d(y,p)+d(y,z)) \leq \frac{1}{2}(d(x,p)-d(x,z))+ \frac{\ell}{2}.
    \end{equation}
    Now, inequality (\ref{almost uniqueness of qg assumption}) and the second inequality of (\ref{almost uniqueness of qg point in qg}) give
    \begin{equation}
    \label{almost uniqueness of qg dif bet gromov prod prev 2}
        d(x,p)-2\delta-3l \leq d(x,y)-d(y,p) \leq d(x,y)-d(y,z)+\ell,
    \end{equation}
    and this together with the triangle inequality imply
    \begin{equation}
    \label{almost uniqueness of qg dif bet gromov prod prev 3}
        d(x,p)-2\delta-3\ell \leq d(x,z)+\ell.
    \end{equation}
    Therefore, inequalities (\ref{almost uniqueness of qg dif bet gromov prod prev}) and (\ref{almost uniqueness of qg dif bet gromov prod prev 3}) yield
    \begin{equation}
    \label{almost uniqueness of qg dif bet gromov prod}
        \langle x,z \rangle_{p}-\langle y,z \rangle_{p} \leq \delta +\frac{5}{2}\ell.
    \end{equation}
    From inequalities (\ref{almost uniqueness of qg delta hyp cond}) and (\ref{almost uniqueness of qg dif bet gromov prod}) we get
    \begin{equation}
    \label{alm uniq qg prod with y in terms of z}
      \langle x,y \rangle_{p} \geq \langle x,z \rangle_{p}-\frac{5}{2}\ell-2\delta.
    \end{equation}
      Inequality (\ref{alm uniq qg prod with y in terms of z}) together with Definition~\ref{gromov prod def} yield
      \[ d(y,p)-d(x,y) \geq d(z,p)-d(x,z)-5\ell-4\delta, \] and thus
      \begin{equation}
      \label{alm uniq qg distances}
         d(y,p)-d(x,y)+d(x,z) \geq d(z,p)-5\ell-4\delta. 
      \end{equation}
      Also, since $z \in \gamma$ we have that \[d(x,z)+d(z,y) \leq d(x,y) + 3\ell,\] and in consequence,
      \begin{equation}
      \label{alm uniq qg z in qg}
          d(x,z)-d(x,y) \leq -d(z,y) + 3\ell.
      \end{equation}
      Therefore, combining inequalities (\ref{alm uniq qg distances}) and (\ref{alm uniq qg z in qg}) we obtain
      \begin{equation}
      \label{alm uniq of qg almost there}
          d(y,p)-d(z,y) \geq d(z,p)-8\ell-4\delta.
      \end{equation}
      The first inequality of (\ref{almost uniqueness of qg point in qg}) and inequalitiy (\ref{alm uniq of qg almost there}) imply that $d(z,p)\leq 8\ell+4\delta$.
\end{proof}

\printbibliography

\vspace{2cm} 

\textbf{Marco Amelio}

Universität Münster

Einsteinstraße 62

48149 Münster, Germany.

E-mail address: \href{mailto:mamelio@uni-muenster.de}{mamelio@uni-muenster.de}

\vspace{5mm}

\textbf{Simon André}

Sorbonne Université and Université Paris Cité

CNRS, IMJ-PRG

F-75005 Paris, France.

E-mail address: \href{mailto:simo.andre@imj-prg.fr
}{simon.andre@imj-prg.fr
}

\vspace{5mm}

\textbf{Katrin Tent}

Universität Münster

Einsteinstraße 62

48149 Münster, Germany.

E-mail address: \href{mailto:tent@wwu.de}{tent@wwu.de}

\end{document}

%% file: Introduction.tex
\section{Introduction}
\label{section introduction}

\let\thefootnote\relax\footnote{The first and third authors, and the second author until September 2023, were supported by the German Research Foundation (DFG) under Germany’s Excellence Strategy EXC 2044–390685587, Mathematics Münster: Dynamics–Geometry–Structure and by CRC 1442 Geometry: Deformations and Rigidity.}
    
Let $n\geq 1$ be an integer, and let $G$ be a group acting on a set $X$ with at least $n$ elements. The action is said to be \emph{sharply n-transitive} if for any two $n$-tuples of distinct elements $(x_{1}, \dots , x_{n})$ and $(y_{1}, \dots, y_{n})$ of $X^{n}$ there exists a unique element $g \in G$ such that $g \cdot x_{i}=y_{i}$ for $1 \leq i \leq n$. A group $G$ is called \emph{sharply $n$-transitive} if there is a set $X$ with at least $n$ elements on which $G$ acts sharply $n$-transitively.

Clearly every group acts sharply 1-transitively on itself by left multiplication. On the other hand, for $n \geq 4$, there are only finitely many sharply $n$-transitive groups, and moreover these groups are necessarily finite and completely classified. Indeed, Jordan proved in \cite{jordan} that the only finite sharply $n$-transitive groups for $n \geq 4$ are the symmetric groups $S_{n}$ and $S_{n+1}$, the alternating group $A_{n+2}$, and the Mathieu groups $M_{11}$ and $M_{12}$ for the cases $n=4$ and $n=5$ respectively. Furthermore, Tits proved in \cite[Chapitre IV, Théorème I]{tits} that there are no infinite sharply $n$-transitive groups for $n \geq 4$. Zassenhaus gave a complete classification of the finite sharply $n$-transitive groups for the cases $n=2$ and $n=3$ in \cite{zassenhaus1} and \cite{zassenhaus2}.  For $n=2$ and $n=3$ there do exist also infinite sharply $n$-transitive groups: for a skew-field $K$, the affine group $AGL(1,K)\cong K_{+} \rtimes K^{*}$ acts sharply 2-transitively on $K$, and for any (commutative) field $K$, the projective linear group $PGL(2,K)$ acts sharply 3-transitively on the projective line. These groups are infinite whenever $K$ is infinite.

Until recently, it was an open question whether a sharply 2-transitive group $G$ necessarily splits in the form $A\rtimes H$ for some non-trivial normal abelian subgroup $A$ (in which case we simply say that $G$ is \emph{split}). The first examples of non-split sharply 2-transitive groups were exhibited by Rips, Segev and Tent in \cite{rips_segev_tent} and by Rips and Tent in \cite{rips_tent}. Then, the first examples of infinite simple sharply 2-transitive groups were constructed by André and Tent in \cite{andre_tent} and by André and Guirardel in \cite{andre_guir_fin_gen_simple} (with additional properties, in particular finite generation).

An important feature associated with a sharply 2-transitive group action is its \emph{characteristic}, which we define as follows. Recall that an element of order 2 in a group is called an \emph{involution}, and a product of two distinct involutions will be called a \emph{translation}. Let $G\curvearrowright X$ be a sharply 2-transitive group action. It is easy to see that $G$ has involutions, and that involutions form a unique conjugacy class. Moreover, if involutions have fixed points, the action $G\curvearrowright X$ is $G$-equivariant to the action of $G$ by conjugation on the set of involutions of $G$, and therefore the translations form a conjugacy class. In this case we define the characteristic of $G\curvearrowright X$ to be the order of a translation if this order is finite (in which case it is necessarily a prime number $\geq  3$), or as $0$  in case this order is infinite. If involutions have no fixed points, we say that the action $G\curvearrowright X$ has characteristic $2$. Thus, sharply $2$-transitive group actions in characteristic $2$ are fundamentally different from those in other characteristics (note in particular that the examples in \cite{rips_segev_tent} show that in characteristic $2$ the action by conjugation of $G$ on itself need not be $2$-transitive on the class of involutions). Notice that when $K$ is a field, the characteristic of $AGL(1,K)\curvearrowright K$ coincides with that of $K$ in the usual sense. We often say that the group $G$ is sharply $2$-transitive of characteristic $p$ without referring to the underlying action $G\curvearrowright X$; this is not ambiguous in characteristic $\neq 2$ since the underlying action is unique (up to $G$-equivariant bijection).\footnote{Note that a group could \emph{a priori} be simultaneously of characteristic 2 and $p\neq 2$ (although there are no known examples).}

The examples in \cite{rips_segev_tent} (in characteristic $2$) and in \cite{rips_tent,andre_tent,andre_guir_fin_gen_simple} (in characteristic $0$) show the class of sharply 2-transitive groups to be wild. By a well-known result due to Kerby (see \cite[Theorem 9.5]{kerby}) every sharply $2$-transitive group in characterstic $3$ splits. On the other hand, the question of the existence of non-split sharply 2-transitive groups in characteristic $p>3$, which we address in this paper for any prime $p$ sufficiently large, had remained completely open until now.

The main difficulty in constructing non-split sharply $2$-transitive groups in characteristic $p>3$ comes from the fact that the methods used so far in order to construct non-split sharply $2$-transitive groups have proceeded through HNN-extensions, which create translations of infinite order. Therefore, since all the translations have order $p$ in a sharply 2-transitive group of characteristic $p>2$, in order to obtain sharply $2$-transitive groups in characteristic $p$, we will have to add new relations of the form $(ij)^p=1$ for distinct involutions $i,j$. This will be achieved by taking small cancellation quotients similar to the quotients used in the solution of the famous Burnside problem about the existence of infinite finitely generated groups of finite exponent (posed by Burnside in \cite{burnside}), in order to guarantee that any translation has order $p$. This makes the construction in odd positive characteristic considerably more complicated than in characteristic 2 and 0. Moreover, it is worth pointing out that the construction of infinite finitely generated groups of even exponent is notoriously more difficult than the construction of infinite finitely generated groups of odd exponent, due to the presence of involutions. As already observed above, sharply 2-transitive groups contain plenty of involutions, and one of the challenges we have to face is to keep these involutions under control when taking small cancellation quotients. We then obtain the following result (see Theorem \ref{embedding theorem} and Corollary \ref{main theorem reformulation} for stronger versions of this statement, in particular the existence of $2^{\aleph_0}$ many pairwise non-isomorphic countable non-split sharply 2-transitive groups in any characteristic $p\geq p'$).

\begin{theorem}
\label{maintheorem}
    There exists a prime number $p'$ such that any group without 2-torsion embeds into a non-split sharply $2$-transitive group of characteristic $p$, for every prime number $p \geq p'$.
\end{theorem}

The methods used in this article, based on Coulon's paper \cite{coulon_1}, will necessarily yield a very large value for $p'$. The existence of non-split sharply 2-transitive groups of characteristic $p$ for small values of $p \geq 5$ remains an open question and seems to bear similarities to the famous open question of the existence of finitely generated infinite groups of small finite exponent.

As a side-remark we note that being a non-split sharply 2-transitive group of characteristic $p\neq 2$ is expressible by a first-order sentence in the language of groups. Hence by the Compactness Theorem, we recover the existence of non-split sharply 2-transitive groups of characteristic 0 as a consequence of Theorem \ref{maintheorem}.\footnote{However,  the strategy of the proof in positive characteristic is based on the strategy used in \cite{rips_tent} in characteristic 0, with additional work to obtain translations of finite order. So this does not actually yield a new proof in characteristic 0.}

Note that \emph{any} group embeds into a non-split sharply  sharply 2-transitive group of characteristic 2 and that \emph{any group not containing 2-torsion} can be embedded into a non-split sharply 2-transitive group of characteristic 0, and even in the centralizer of an involution in such a group (see \cite{andre_tent}). Theorem \ref{maintheorem} shows that a similar phenomenon occurs in (large) prime characteristic.

The proof of our main result has its roots in \cite{rips_tent}, but a new step has to be added in the iterative construction in order to solve the difficulty mentioned above: given a group action with trivial 2-point stabilizers we inductively consider HNN-extensions to obtain a sharply $2$-transitive action, but the main point of the paper is concerned with the difficulty of taking the appropriate quotients to obtain positive characteristic $p$ by killing $p$-powers of products of distinct involutions of infinite order, for some uniform prime $p$.

\subsection*{Structure of the paper}In Section \ref{Sec:Outline} we define two classes of groups that form the framework for our construction, and we give the proof of Theorem \ref{maintheorem} (and more generally of Theorem \ref{embedding theorem} and Corollary \ref{main theorem reformulation}), assuming two technical propositions whose proofs are postponed to Sections \ref{section small cancellation quotient} and \ref{section second quotient}. In Section \ref{section hyp spaces and grp act} we recall background on hyperbolic metric spaces and on actions of groups by isometries on these spaces. Although most of the definitions and results given in this section are well-known in the context of geodesic metric spaces, we have to work in the more general context of length metric spaces, and we explain how to adapt the proofs when necessary (see also Appendix \ref{appendix}). Then, in Section \ref{subsection geometric small cancellation}, we introduce the geometric small cancellation methods that will be crucial to prove Theorem \ref{maintheorem}. These geometric small cancellation methods were introduced by Gromov in \cite{gromov_mesoscopic}, and further developed by Delzant and Gromov in \cite{delzant_gromov}, and then notably by Arzhantseva and Delzant in \cite{arzhantseva_delzant}, by Coulon in \cite{coulon_3, coulon_2, coulon_1}, by Cantat and Lamy in \cite{cantat_lamy} (to prove that the Cremona group is not simple) and by Dahmani, Guirardel and Osin in \cite{dahmani_guirardel_osin}; see also \cite{coulon_bourbaki}. These methods proved useful to create various exotic quotients of groups exhibiting some kind of negative curvature, among which Burnside quotients of hyperbolic groups (i.e.\ infinite quotients of hyperbolic groups of finite exponent). In particular, these methods were used to answer the Burnside problem, originally solved by Adian and Novikov in \cite{adian_novikov} by means of complicated combinatorial techniques (see also \cite{adian, olshanskii, ivanov, lysenok, delzant_gromov, coulon_2, coulon_4, atkarskaya_rips_tent}).

\subsection*{Acknowledgement} For the Burnside quotients we use results from \cite{coulon_1} and we are very grateful to  R\'emi Coulon for several helpful conversations and for his enthusiasm for this project.

%% file: Preliminaries.tex
\section{Preliminaries and outline of the proof}
\label{Sec:Outline}

We start by fixing some notation and terminology. Recall that for a group $G$, an \emph{involution} is an element $r \in G$ of order two, and a \emph{translation} is a product $rs \in G$ of two distinct involutions $r$ and $s$. For an element $g \in G$, we write $\lvert g \rvert$ for its order. We will use the following convention: for $g,h\in G$, the conjugate of $h$ by $g$ is $g^{-1}hg$. In this paper, a pair of elements or of subgroups of a group $G$ always refers to an ordered pair.

\begin{notation}
Let $G$ be a group and $\mathcal{S} \subseteq G$ an arbitrary subset.
\begin{itemize}
    \item We write $I(\mathcal{S})=\{ r : r \in \mathcal{S}, \ \lvert r \rvert =2 \}$ for the set of involutions of $\mathcal{S}$.
    \item We write $Tr(\mathcal{S}) = \{ rs : rs \in \mathcal{S}, \ r,s \in I(G), \  r \neq s \}$ for the set of translations of $\mathcal{S}$.
    \item We write $I(\mathcal{S})^{(2)}=\{ (r,s) : r,s \in I(\mathcal{S}), \ r\neq s \}$ for the set of (ordered) pairs of distinct involutions of $\mathcal{S}$. 
    \item For involutions $r \neq s$ of $G$, we write $D_{r,s}= \langle r,s \rangle$ (note that this group is dihedral of order $2\lvert rs \rvert$, possibly infinite).
    \item We write $O(G)$ for the set of natural numbers that are orders of elements of $G$, that is, \[O(G)=\{ n \in \mathbb{N} \, : \, \exists g \in G \, , \, \lvert g \rvert=n \}.\]
\end{itemize}
Note that the sets $I(G)$ and $Tr(G)$ are invariant under conjugation. Note also that $I(\mathcal{S})=I(G)\cap \mathcal{S}$ and $Tr(\mathcal{S})=Tr(G)\cap \mathcal{S}$.
\end{notation}

Given a sharply 2-transitive group action $G\curvearrowright X$, it is easy to see that the set $I(G)$ is non-empty: for any pair of distinct points $x,y \in X$, there is an element $r$ of $G$ swapping $x$ and $y$. Since $r^{2}$ fixes both $x$ and $y$, by the uniqueness condition we must have $r^{2}=1$. Similarly, it is easy to see that all involutions are conjugate. In particular,  either every involution has a (necessarily unique) fixed point, or no involution has a fixed point.

If  $G$ acts sharply 2-transitively on a set $X$ such that involutions of $G$ have fixed points, there is a $G$-equivariant bijection $I(G) \longrightarrow X$ given by $r \longmapsto \text{Fix}(r)$, where $G$ acts on $I(G)$ by conjugation (see e.g.\ Lemma 3.1 in \cite{tent}).

Hence, if involutions in a sharply 2-transitive group have fixed points, we may assume that the group acts by conjugation on its set of involutions. An immediate consequence of this lemma is the fact that in a sharply 2-transitive group where involutions have fixed points, all translations are conjugate. Recall that in this case we define the characteristic of a sharply 2-transitive group in relation to the order of translations (we refer to the introduction).

Notice that the characteristic of a sharply 2-transitive group is necessarily 0 or a prime number $p$ (since otherwise the subgroup generated by a translation $t$ of order $m \cdot m'$ with $ m,m'>1$ would contain a translation of order $m$ that cannot be conjugate to $t$) and the case $p=2$ occurs if and only if involutions are fixed-point free (since two involutions with different fixed points cannot commute).

From now on, for a prime number $p$ we will write $\Fp$ for the finite field with $p$ elements. Also, for $n \geq 2$ we will write $C_{n}$ for the cyclic group of order $n$ and $D_{n}$ for the dihedral group of order $2n$, that is, the group given by the presentation \[ D_{n}= \langle r,s \ \vert \  r^{2}=s^{2}=(rs)^{n}=1 \rangle \cong C_n\rtimes C_2. \]



We now introduce two classes, denoted $\classc$ and $\classcprime$, which will be used in the construction of groups as in Theorem \ref{maintheorem}. The conditions for the class of groups $\classc$ are comparable to those for the classes considered in \cite{rips_tent,andre_tent,andre_guir_fin_gen_simple}, but in order to take small cancellation quotients as explained in the introduction, we will introduce a class $\classcprime$ whose elements are pairs composed of a group $G$ and a tree $X$ endowed with an action of $G$ satisfying a number of technical assumptions that will allow us to apply the results from \cite{coulon_1}.

\begin{notation}
    Given a subgroup $H$ of a group $G$, we  write $N_{G}(H)$ for the normalizer of $H$ in $G$, and $\text{\text{Cen}}(g)$ for the centralizer of an element $g$ in $G$.
\end{notation}

\begin{definition}
\label{def paff and pmin}
    Let $G$ be a group, $(r,s) $ a pair in $ I(G)^{(2)}$, $p$ an odd prime number.

    \begin{itemize}
        \item We say that $(r,s)$ is a pair of \emph{$p$-minimal type} if $\lvert rs \rvert=p$ and $N_{G}(\langle rs \rangle)=D_{r,s}$.
        \item We say that $(r,s)$ is a pair of \emph{$p$-affine type} if $D_{r,s}$ is contained in a subgroup $H$ of $G$ isomorphic to $\agl$. 
    \end{itemize}
\end{definition}




 \begin{remark}
\label{conj of pmin and paff is pmin paff}It is not hard to see that $\agl$ has a unique subgroup $D$ isomorphic to $D_p$, and that $I(D)=I(\agl)$. Therefore, the isomorphism between $H$ (as in Definition \ref{def paff and pmin}) and $\agl$ induces an isomorphism between $D_{r,s}$ and $D$. It follows that $H$ normalizes $D_{r,s}$ and that $H$ acts sharply 2-transitively on $I(D_{r,s})=I(H)$ by conjugation. Note in particular that $\lvert rs \rvert=p$. Moreover, if $(r,s) \in I(G)^{(2)}$ is of $p$-affine type and no nontrivial element of $G$ centralizes both $r$ and $s$, then in fact $N_{G}(D_{r,s})$ coincides with $H$, and thus is isomorphic to $\agl$, since the subgroup $H\subseteq N_{G}(D_{r,s})$ acts already 2-transitively by conjugation on $I(D_{r,s})$ .
\end{remark}

\begin{remark}
 Clearly, if  $(r,s) \in I(G)^{(2)}$ is of $p$-minimal type (respectively, of $p$-affine type), then so is every conjugate of $(r,s)$.
 \end{remark}

\begin{definition}
\label{definition anss2t}
    Let $G$ be a group, $p$ be an odd prime number. We will say that $G$ is \emph{almost sharply 2-transitive of characteristic $p$} if it satisfies the following conditions.
    \begin{enumerate}[label=(\arabic*)]
        \item Every translation is either of order $p$ or of infinite order, and every pair $(r,s) \in I(G)^{(2)}$ such that $rs$ is of order $p$ is either of $p$-minimal type or of $p$-affine type.
        \label{def class C pairs of minimal or affine type}
        \item The (normal) subgroup $N=\langle Tr(G)\rangle$ contains no involution.
        \label{def class C N without 2-torsion}
        \item The set of pairs $(r,s) \in I(G)^{(2)}$ of $p$-affine type is non-empty and $G$ acts transitively on it by conjugation.\label{def class C G trans on affine trans}
        \item For every pair $(r,s) \in I(G)^{(2)}$ the subgroup $\text{\text{Cen}}(rs)$ is cyclic  and generated by a translation.\label{def class C cent of trans is cyclic}
    \end{enumerate}
\end{definition}

\begin{remark}
\label{no elem cent dist invol in c cprime}If $G$ is an almost sharply 2-transitive group of of characteristic $p$, Condition \ref{def class C G trans on affine trans}  implies that $I(G)$ is non-empty. Condition \ref{def class C pairs of minimal or affine type} implies that no two distinct involutions commute (that would yield a translation of order 2). In addition, by Condition \ref{def class C cent of trans is cyclic} no non-trivial element of $G$ centralizes two distinct involutions: suppose one such $g\in G$ centralizes involutions $r\neq s\in I(G)$. Then $g\in \Cen(rs)=\langle h\rangle$ for some $h\in G$.  If $rs$ has order $p$, we can choose $h=rs$ and we see that $rgr^{-1}=g^{-1}\neq g$, a contradiction. Otherwise $rs$ has infinite order and we have  $g^m\in\langle rs\rangle\subset \Cen(rs)$ for some $m\in\mathbb{Z}^*$, so again $rg^mr^{-1}=g^{-m}\neq g^m$, a contradiction. 
\end{remark}


We now introduce the concept of a quasi-malnormal pair of subgroups of a given group, and we relate in Lemma \ref{qm pair} this concept to the case where the subgroups in consideration are finite dihedral of order $2p$.

\begin{definition}
\label{def quasi-malnormal}Let $G$ be a group. A subgroup $K$ of $G$ is said to be \emph{quasi-malnormal} if for all $g \in G \setminus K$ we have $\lvert K \cap g^{-1}Kg \rvert \leq 2$. A pair of subgroups $(K,K')$ of $G$ is said to be \emph{quasi-malnormal} if $K$ is quasi-malnormal and for all $g \in G$ we have $\lvert K \cap g^{-1}K'g \rvert \leq 2$.
\end{definition}

\begin{lemma}
\label{qm pair}Let $G$ a group. Let $(r,s) \in I(G)^{(2)}$ be a pair of $p$-minimal type and $(r',s') \in I(G)^{(2)}$ be a pair of $p$-affine type. Then the pair $( D_{r,s},D_{r',s'})$ is quasi-malnormal.\end{lemma}

\begin{proof}We first prove that $ D_{r,s}$ is quasi-malnormal. Suppose that there is some $g \in G \backslash D_{r,s}$ such that $\lvert D_{r,s} \cap (g^{-1}D_{r,s} g) \rvert > 2$. Then, this intersection must have order $p$ or $2p$, and in any case both subgroups have the same translations, that is, $\langle rs \rangle = g^{-1} \langle rs \rangle g$. Namely, $g$ normalizes $\langle rs \rangle $, contradicting the assumption that $(r,s)$ is a pair of $p$-minimal type.

Then, let us prove that $D_{r,s} \cap (g^{-1}D_{r',s'}g)$ has order at most $2$. Assume towards a contradiction that $\lvert D_{r,s} \cap g^{-1}D_{r',s'}g \rvert > 2$. Since a conjugate of a pair of $p$-affine type is still of $p$-affine type, we can assume without loss of generality that $g=1$. Then $\langle rs \rangle = \langle r's' \rangle$. But $\langle r's' \rangle$ is a characteristic subgroup of $D_{r',s'}$, so the subgroup $H \cong \agl$ of $N_{G}(D_{r',s'})$ also normalizes $\langle r's' \rangle= \langle rs \rangle$, contradicting the assumption that $(r,s)$ is a pair of $p$-minimal type.\end{proof}

\begin{remark}
In fact, a pair $(r,s) \in I(G)^{(2)}$ is of $p$-minimal type if and only if the subgroup $D_{r,s}$ is quasi-malnormal. Indeed, if there exists some $g \in N_{G}(\langle rs \rangle) \backslash D_{r,s}$, then $ D_{r,s} \cap (g^{-1}D_{r,s} g) $ has order at least $p$ and thus $D_{r,s}$ is not quasi-malnormal.
\end{remark}

Now we are ready to introduce the class $\classc$.

\begin{definition}[Class $\classc$]
\label{definition class C}
    Let $G$ be a group. We will say that $G$ belongs to the class $\classc$ if it is  almost sharply 2-transitive of characteristic $p$ with all translations of order $p$ (i.e.\ no translation has infinite order).
\end{definition}

\begin{remark}
\label{shapr 2trans in terms of affine type}
Clearly $\agl\in\classc$, so in particular this class is non-empty. Furthermore, if $G$ is in $\classc$ and every pair $(r,s) \in I(G)^{(2)}$ is of $p$-affine type, then $G$ is sharply $2$-transitive of characteristic $p\neq 2$.  
 \end{remark}
 
Our construction of non-split sharply 2-transitive groups starts from a group in $\classc$ and proceeds inductively by using HNN-extensions to conjugate pairs of $p$-minimal type to pairs of $p$-affine type. Furthermore, we will prove that the sharply 2-transitive groups that we construct contain non-commuting translations. The following theorem implies that these groups will not split. 

\begin{theorem}[see \cite{neumann}]
\label{if splits trans abelian}
    Let $G$ be a sharply 2-transitive group. Then $G$ splits if and only if $Tr(G)\cup\{ 1\}$ is an abelian subgroup (which is normal since the set of translations is invariant under conjugation).
\end{theorem} 

The cornerstone of the inductive process is the following result.

\begin{proposition}
\label{thm: classes stable under hnn and free products}
    Let $p$ be an odd prime number and let $G$ be an almost sharply 2-transitive group of characteristic $p$.
\begin{enumerate}
\item  Let $(r,s), (r',s') \in I(G)^{(2)}$ with $(r,s)$ of  $p$-affine type, and $(r',s')$ of $p$-minimal type, so $D_{r,s}$ and $D_{r',s'}$ are isomorphic to $D_p$.
Then the HNN-extension \[\gstar = \langle G,t \ \vert \ t^{-1}rt=r', t^{-1}st=s'\rangle\] is almost sharply 2-transitive of characteristic $p$.
\item Let $H$ be a non-trivial group without $2$-torsion. Then $\gstar= G * H$ is almost sharply 2-transitive  of characteristic $p$.
\end{enumerate}    
\end{proposition}


However, in both cases of Proposition \ref{thm: classes stable under hnn and free products}, even if $G\in\classc$, we have $\gstar\notin\classc$ since $\gstar$ contains translations of infinite order. Thus, in order to eventually obtain a sharply 2-transitive group of characteristic $p$, in particular a group where all translations have order $p$, we need to take partial periodic quotients (see Theorem \ref{res : SC - partial periodic quotient}) of the (normal) subgroup generated by translations. To that purpose, we now introduce a new class $\classcprime$. Definition \ref{definition class Cprime} below involves Theorem \ref{res : SC - partial periodic quotient}, which is stated together with the necessary background in Section~\ref{subsection geometric small cancellation} following work of Coulon (see \cite{coulon_1}).


\begin{definition}[Class $\classcprime$]
\label{definition class Cprime}
Let $G$ be a group acting by isometries on a (simplicial) tree $X$. Write $N= \langle Tr(G) \rangle$ and $Q=Tr(G)$. We say that $(G,X)$ belongs to the class $\classcprime$ if $G$ is almost sharply 2-transitive of characteristic $p$ and the following conditions are satisfied.
    \begin{enumerate}[label=(\arabic*')]
        \item The action of $G$ on $X$ is non-elementary and acylindrical (see Definition \ref{acylindr definition}), and thus it is WPD (see Definition \ref{definition wpd}).
        \label{def class Cprime no parabolic}
        \item The action is such that $e(N,X)=1$, $r_{\text{inj}}(Q,X) \geq 1$, $\nu(N,X) \leq 5$ and $A(N,X)=0$ (see Subsection \ref{subsubsection invariants} for the definitions of the parameters), and the prime number $p$ is at least $n_1$ (the constant provided by Theorem \ref{res : SC - partial periodic quotient} for these values of the parameters and hyperbolicity constant $\delta=0$).
        \label{def class Cprime satisfies the ind lemma}
        \item Every translation of infinite order is loxodromic for its action on $X$.
        \label{def class Cprime inf trans are loxodromic}
        \item A finite subgroup of $G$ normalized by a loxodromic element has order at most 2.
        \label{def class Cprime no subg of odd ord norm by lox}
    \end{enumerate}
\end{definition}

We will sometimes say that $G$ belongs to the class $\classcprime$ with no mention of the space $X$ if this is not necessary. We will prove in Section \ref{section small cancellation quotient} that a group in class $\classcprime$ in fact satisfies the hypotheses of Theorem \ref{res : SC - partial periodic quotient} (that is, that the family $Q$ is strongly stable, see Definition \ref{def strongly stable family}).
 
We will show in Section \ref{section second quotient} the following refinement of Proposition~\ref{thm: classes stable under hnn and free products}.

\begin{proposition}
\label{prop:ext in cprime}
If $G\in\classc$, then for $G^*$ as in Proposition~\ref{thm: classes stable under hnn and free products} we have $(\gstar,X)\in\classcprime$ where $X$ is the Bass-Serre tree associated with the decomposition of $G^*$ as an HNN-extension or as a free product.
\end{proposition}

This allows us to take appropriate quotients to return to $\classc$. Namely, we will show in Section \ref{section small cancellation quotient} the following result.

\begin{proposition}
\label{prop: classes stable under pp quotients}
    If $(G, X)$  belongs to $\classcprime$, then $G$ has a quotient $\bar G\in\classc$ such that the images of pairs of distinct involutions of $p$-affine type (respectively, of $p$-minimal type) will again be of $p$-affine type (respectively, of $p$-minimal type), any subgroup of $G$ that is elliptic for the action on $X$ projects to an isomorphic image in $\bar{G}$, and $O(G)=O(\bar{G})$.
\end{proposition}

We will need the following easy observation.

\begin{remark}
\label{class cprime stable under un of inf chains}
    Notice that if a group $G$ is the union of an infinite chain of subgroups $H_{m}$, $m \in \mathbb{N}$, all of them belonging to the class $\classc$, then so does $G$. Moreover, if for all $m,m' \in \mathbb{N}$ we have $O(H_{m})=O(H_{m'})$, then $O(G)=O(H_{0})$.
\end{remark}

Propositions~\ref{prop:ext in cprime} and~\ref{prop: classes stable under pp quotients} are the tools used in the proof our main result, i.e. the following stronger version of Theorem \ref{maintheorem}.

\begin{theorem}[The Embedding Theorem]
\label{embedding theorem}
    There exists a prime number $p'$ with the following property. If $G\in\classc$ for a prime number $p \geq p'$ and if $H$ has no 2-torsion, then $G$ and $H$ embed into a non-split sharply 2-transitive group $\mathbf{G}$ of characteristic $p$ with $O(\mathbf{G})=O(G) \cup O(H)$ and of cardinality equal to the maximum of $\aleph_{0}$ and the cardinalities of $G$ and $H$.
\end{theorem}





\begin{proof}
Let $n_1$ be the constant involved in Definition \ref{definition class Cprime}, Condition \ref{def class Cprime satisfies the ind lemma}. Let $p'$ be the smallest odd prime number such that $p'\geq n_1$, and let $p \geq p'$ be a prime number. Let $G\in\classc$ and $H$ be a group without 2-torsion. By Proposition~\ref{prop:ext in cprime} applied to $G^*=G*H$ we see that $(G^*,X)\in\classcprime$ where $X$ is the Bass-Serre tree associated with the free product. Clearly, we have $O(\gstar)=O(G) \cup O(H)$.

Let $G_{0}^{0}=\overline{G^*}\in\classc$ be the group obtained from $(G^*,X)$ by applying Proposition~\ref{prop: classes stable under pp quotients}. In particular, we have $O(\gstar)=O(G_{0}^{0})$. Since $G$ and $H$ are elliptic for their action on the Bass-Serre tree of the free product, $G_{0}^{0}$ contains isomorphic images of these groups. Furthermore, we will prove that $G_{0}^{0}$ contains non-commuting translations (see Lemma \ref{non-commuting translations}).

Now, fix a pair of distinct involutions $(i,j) \in G_{0}^{0}$ of $p$-affine type. Enumerate all pairs of $I( G_{0}^{0})^{(2)}$ as $(i_{0}^{\lambda},j_{0}^{\lambda}) : \lambda < \gamma$. We will construct inductively $G_{0}^{\alpha}$ for $\alpha < \gamma$. For successor ordinals, suppose $G_{0}^{\alpha}$ has already been constructed, that this group is in class $\classc$, that $G_{0}^{0}$ embeds into $G_{0}^{\alpha}$, and that we have $O(G)=O(G_{0}^{\alpha})$. Consider the pair $(i_{0}^{\alpha},j_{0}^{\alpha})$. If this pair is of $p$-affine type, we set $G_{0}^{\alpha+1}=G_{0}^{\alpha}$. If the pair is of $p$-minimal type, we set \[(G_{0}^{\alpha})^{*}= \langle G_{0}^{\alpha} \, , \, t \ \vert \ t^{-1}jt=j_{0}^{\alpha} \, , \, t^{-1}it=i_{0}^{\alpha} \rangle.\]Note that this is a well-defined HNN-extension since $ D_{i,j}$ and $D_{i_{0}^{\alpha},j_{0}^{\alpha}}$ are isomorphic to $D_p$. Clearly, in $(G_{0}^{\alpha})^{*}$ the pair $(i_{0}^{\alpha},j_{0}^{\alpha})$ is of $p$-affine type.

By Proposition \ref{prop:ext in cprime}, the pair $((G_{0}^{\alpha})^{*},X)$ is in $\classcprime$, with $X$ the Bass-Serre tree of the HNN-extension. Moreover, clearly $G_{0}^{0}$ embeds into $(G_{0}^{\alpha})^{*}$, and $O((G_{0}^{\alpha})^{*})=O(G_{0}^{\alpha})$.

With this in mind, we set $G_{0}^{\alpha+1}=\overline{(G_{0}^{\alpha})^{*}}$ as in Proposition \ref{prop: classes stable under pp quotients}. This group belongs to the class $\classc$, and is such that $O(G_{0}^{\alpha+1})=O(G_{0}^{\alpha})$. Furthermore $G_{0}^{0}$ embeds in $G_{0}^{\alpha+1}$ (and therefore, so do $G$ and $H$) since it has an isomorphic image in the HNN-extension that is elliptic for its action on the Bass-Serre tree $X$. Moreover, every pair $(i_{0}^{\lambda},j_{0}^{\lambda})$ of distinct involutions is of $p$-affine type for $\lambda \leq \alpha$.

For a limit ordinal $\alpha$, we set $G_{0}^{\alpha}=\cup_{\beta< \alpha} G_{0}^{\beta}$. By Remark \ref{class cprime stable under un of inf chains} this group is in class $\classc$, $G_{0}^{0}$ embeds into it, and is such that $O(G_{0}^{\alpha})=O(G_{0}^{0})$. Moreover, every pair $(i_{0}^{\lambda},j_{0}^{\lambda})$ of distinct involutions is of $p$-affine type for $\lambda < \alpha$.

Put $G_{1}^{0}=\cup_{\beta< \gamma} G_{0}^{\beta}$. With Remark \ref{class cprime stable under un of inf chains} in mind, we can deduce similar properties to the ones in the previous paragraph, and in this case, every pair $(r,s) \in I(G_{0}^{0})^{(2)}$ is of $p$-affine type. The cardinality of this group is clearly the maximum of $\aleph_{0}$ and the cardinalities of $G$ and $H$.

Now, build $G^{0}_{k+1}$ from $G^{0}_{k}$ in a completely analogous way to the construction of $G^{0}_{1}$ from $G^{0}_{0}$, just replacing $G^{\alpha}_{0}$ by $G^{\alpha}_{k}$, $G^{\alpha+1}_{0}$ by $G^{\alpha+1}_{k}$, enumerating the pairs of $I(G_{k}^{0})$ as $(i_{k}^{\lambda},j_{k}^{\lambda}) : \lambda < \gamma$, and conjugating pairs of $p$-minimal type to the pair $(i,j)$ of $p$-affine type. Once again, assume that the group $G_{k}^{\alpha}$ is in class $\classc$, with $G_{k}^{0}$ (and thus also $G_{0}^{0}$) embedding into it, and with $O(G_{k}^{\alpha})=O(G_{0}^{0})$. If we need to take an HNN-extension at that step, Proposition \ref{prop:ext in cprime} ensures that the pair $((G_{k}^{\alpha})^{*},X)$, with $X$ the Bass-Serre tree of the extension, belongs to $\classcprime$. Moreover, Proposition \ref{prop: classes stable under pp quotients} gives that $G_{k}^{\alpha+1}$ belongs to $\classc$ with $O(G_{k}^{\alpha+1})=O(G_{0}^{0})$, and with a completely analogous argument to the case $k=0$ we get that $G_{0}^{0}$ embeds into this group. Finally, Remark \ref{class cprime stable under un of inf chains} guarantees that $G_{k+1}^{0}$ is also in class $\classc$ and that $O(G_{k+1}^{0})=O(G)$. By construction, $G_{0}^{0}$ embeds into $G_{k+1}^{0}$ and all pairs $(r,s) \in I(G_{k}^{0})^{(2)}$ are of $p$-affine type in this group, since they are conjugate to one such pair $(i,j)$. Again, the cardinality of $G_{k+1}^{0}$ is the maximum of $\aleph_{0}$ and the cardinalities of $G$ and $H$.

Set now $\mathbf{G}= \cup_{k< \omega} G_{k}^{0}$. Since every group $G_{k}^{0}$ is in class $\classc$ and is such that $O(G_{k}^{0})=O(G_{0}^{0})$, then by Remark \ref{class cprime stable under un of inf chains} $\mathbf{G}$ is itself in class $\classc$ with $O(\mathbf{G})=O(G_{0}^{0})$. Moreover, every pair $(r,s) \in I(\mathbf{G})^{(2)}$ is of $p$-affine type: if this pair of involutions appears for the first time in some $G_{m}^{0}$, then our construction ensures that this pair is of $p$-affine type in $G_{m+1}^{0}$. In particular, $\mathbf{G}$ is sharply 2-transitive of characteristic $p$. By construction, $G_{0}^{0}$ (and thus, $G$ and $H$) embeds into $\mathbf{G}$, so it contains non-commuting translations and thus by Theorem \ref{if splits trans abelian} it is not split. Finally, the cardinality of this group is the maximum of $\aleph_{0}$ and the cardinalities of $G$ and $H$. Thus, $\mathbf{G}$ is a group as claimed by Theorem \ref{embedding theorem}.
\end{proof}

Theorem \ref{embedding theorem} also implies the following result. 

\begin{corollary}[The Existence Theorem]
\label{main theorem reformulation}
    There exists a prime number $p'$ such that for every prime number $p \geq p'$ there are $2^{\aleph_{0}}$ many pairwise non-isomorphic countable non-split sharply 2-transitive groups of characteristic $p$.
\end{corollary}

\begin{proof}

Keeping the same notation as the statement of Theorem \ref{embedding theorem}, put $G= \agl \in \classc$. We will show how to choose groups $H$ without 2-torsion appropriately so as to obtain $2^{\aleph_{0}}$ many non-isomorphic such groups. Denote by $P$ the set of prime numbers greater than $p$. There are $\aleph_{0}$ many such primes, and therefore there are $2^{\aleph_{0}}$ many subsets of $P$. For every $S \subseteq P$, write $H_{S}$ for the direct sum of the cyclic groups of order $s \in S$. Notice that the prime numbers in $O(H_{S})$ are precisely the ones in $S$. Moreover, by the choice of $P$, none of these primes will be in $O(\agl)$, since $\lvert \agl \rvert=p(p-1)$, and thus the greatest prime number in $O(\agl)$ is $p$.

Consider now the group $\mathbf{H}_{S}$ provided by Theorem \ref{embedding theorem} with $H=H_{S}$. Such a group is a countable non-split sharply 2-transitive group of characteristic $p$ with $O(\mathbf{H}_{S})=O(\agl) \cup O(H_{S})$. In particular, we have $O(\mathbf{H}_{S}) \cap P=S$. Therefore, if $S' \neq S''$ are subsets of $P$, $\mathbf{H}_{S'} \ncong \mathbf{H}_{S''}$ (since $O(\mathbf{H}_{S'}) \neq O(\mathbf{H}_{S''})$). Since there are $2^{\aleph_{0}}$ many distinct subsets of $P$, this yields the desired conclusion.
\end{proof}

The remainder of this article is devoted to recalling the necessary background and developing the tools to complete the proof of Theorem \ref{embedding theorem}, that is, to proving Propositions \ref{prop:ext in cprime} and \ref{prop: classes stable under pp quotients}.